\documentclass[10pt]{paper}%
\usepackage[T1]{fontenc}
\usepackage{amsmath,amssymb}
\usepackage{fullpage}
\usepackage{amsthm}
\usepackage{amsmath}
\usepackage{graphicx}
  \usepackage[all]{xy}
\usepackage{amsfonts}
\usepackage{amssymb}%
\providecommand{\U}[1]{\protect\rule{.1in}{.1in}}
\providecommand{\U}[1]{\protect\rule{.1in}{.1in}}
\providecommand{\U}[1]{\protect\rule{.1in}{.1in}}
\providecommand{\U}[1]{\protect\rule{.1in}{.1in}}
\providecommand{\U}[1]{\protect\rule{.1in}{.1in}}
\newcommand{\ulambda}{{\boldsymbol{\lambda}}}
\newcommand{\urho}{{\boldsymbol{\rho}}}
\newcommand{\utau}{{\boldsymbol{\tau}}}

\newtheorem{Th}{Theorem}[section]
\newtheorem{lemma}[Th]{Lemma}
\newtheorem{Cor}[Th]{Corollary}
\newtheorem{Prop}[Th]{Proposition}

\theoremstyle{definition}
\newtheorem{Rem}[Th]{Remark}{\rmfamily}
\theoremstyle{definition}
\newtheorem{Def}[Th]{Definition}{\rmfamily}
\newtheorem{Exa}[Th]{Example}{\rmfamily}

\newcommand{\Z}{\mathbb{Z}}

\newcommand{\C}{\mathbb{C}}
\newcommand{\Comp}{\operatorname{Comp}} 
\newcommand{\Mat}{\operatorname{Mat}}
\newcommand{\cH}{\mathcal{H}}

\newcommand{\mS}{\mathfrak{S}}
\newcommand{\tr}{\operatorname{tr}}
\newcommand{\FG}{{\rm{F}}\Gamma}
\newcommand{\FD}{{\rm{F}}\Delta}
\newcommand{\M}{{\rm M}}
\newcommand{\tT}{\widetilde{T}}
\newcommand{\tg}{\widetilde{g}}

\begin{document}
\title{An isomorphism Theorem for Yokonuma--Hecke algebras and applications to link invariants}
\author{ N. Jacon and L. Poulain d'Andecy}

\maketitle
\date{}

\begin{abstract}
 We develop several applications of the fact that  the Yokonuma--Hecke algebra of the general linear group GL is isomorphic to  a direct sum of matrix algebras associated to Iwahori--Hecke algebras of type $A$. This includes a description of the semisimple and modular representation theory of the Yokonuma--Hecke algebras of GL and a complete classification of all the Markov traces for them. Finally, from these Markov traces, we construct 3-variables polynomials which are invariants for framed and classical knots and links, and investigate their properties with the help of the isomorphism. 
\end{abstract}


\section{Introduction}

\paragraph{1.1.} The Yokonuma--Hecke algebras have been introduced by Yokonuma in \cite{Yo} as  centraliser algebras of  the permutation representation of  Chevalley groups  $G$ with respect to a maximal unipotent subgroup of $G$. They are  thus  particular cases of unipotent Hecke algebras and they admit a natural basis indexed by double cosets (see \cite{th2} for more details on general unipotent Hecke algebras). For the Yokonuma--Hecke algebras, the natural description has been transformed into a  simple presentation with generators and relations \cite{ju,juka}.
Assume that $q$ is a power of a prime number then, from this presentation, one can observe that the Yokonuma--Hecke algebra of $G=\text{GL}_n (\mathbb{F}_{q})$ (sometimes called the Yokonuma--Hecke algebra of type $A$)    is a deformation of the group algebra of the complex reflection group of type $G(d,1,n)$, where $d=q-1$. 

In this paper, we will consider the generic Yokonuma--Hecke algebras $Y_{d,n}$ ($n\in\Z_{\geq0}$) depending on two indeterminates $u$ and $v$ and a positive integer $d$, over the ring $\C[u^{\pm1},v]$. The algebra $Y_{d,n}$ is also a deformation of the group algebra of the complex reflection group of type $G(d,1,n)$, for any $d$, and the Yokonuma--Hecke algebra of $\text{GL}_n (\mathbb{F}_{q})$ is obtained by considering the specialization $u^2=q$, $v=q-1$ and the case $d=q-1$.

There exist others well-known deformations of complex reflection groups of type $G(d,1,n)$ which have been intensively studied during the last  past years  : the Ariki--Koike algebras. These algebras turn out to have a deep representation theory  (in both semisimple and modular cases) which is now quite well-understood (see for example \cite{GJ} for an overview). 

\vskip .1cm
Generalizations of Yokonuma--Hecke algebras associated to Chevalley groups appeared in the work of G. Lusztig on character sheaves \cite{lu}. In this work, it is proved that these algebras are isomorphic to certain direct sums of matrix algebras over classical Iwahori--Hecke algebras. For the Yokonuma--Hecke algebras $Y_{d,n}$ under consideration here, it turns out that the Iwahori--Hecke algebras appearing in the matrix algebras are all of type A. %

Coherently with this result, the set of simple modules for $Y_{d,n}$ has been explicitly described in the semisimple situation in combinatorial terms in \cite{CPA1} (see also \cite{th1} for general results on the semisimple representation theory of unipotent Hecke algebras), and a criterion of semi simplicity has been deduced. In addition, a certain  symmetrizing form has been defined and the associated Schur elements (which control a part of the representation theory of the algebra) have been calculated directly. They appear to have a particular simple form, namely products of Schur elements of Iwahori--Hecke algebras of type $A$. This suggests an interpretation of this symmetrizing form and its Schur elements with the help of the isomorphism. %

\vskip .1cm
In another way, a  motivation for studying  the Yokonuma--Hecke  algebra comes from topology and more precisely the theory of 
 knot and link invariants. Indeed, the algebra $Y_{d,n}$ is naturally a quotient of the framed braid group algebra, and in turn can be used to search for isotopy invariants for framed links in the same spirit as the Iwahori--Hecke algebra of type $A$ is used to obtain an invariant for classical links (the HOMFLYPT polynomial).
  
In \cite{ju2}, Juyumaya introduced on $Y_{d,n}$ an analogue of the Ocneanu trace of the Iwahori--Hecke algebra of type $A$. This trace was subsequently used by Juyumaya and Lambropoulou to produce isotopy invariants for framed links \cite{jula0,jula3}. Remarkably, they also produced  isotopy invariants for classical links and singular links \cite{jula1,jula2}. Even though the obtained invariants for classical links are different from the HOMFLYPT polynomial (excepted in some trivial cases), all the computed examples seem to indicate that the invariants for classical links obtained from $Y_{d,n}$ so far are topologically equivalent to the HOMFLYPT polynomial \cite{chla}. 
 In fact, if we restrict to  classical knots,  such an equivalence has been announced in \cite{vous}. 

Again, a natural strategy is to seek for an explanation of this fact starting from the isomorphism connecting the algebra $Y_{d,n}$ with Iwahori--Hecke algebras of type $A$. %
 
\paragraph{1.2.} In this paper, we give  several answers and results in both directions:
  the representation theory and the knots and links theory.  After recalling several results and detailing the structure of the algebras (in Section \ref{sec-def}), in the third section, we provide explicit formulas for an isomorphism, over the ring $\C[u^{\pm1},v]$, between the Yokonuma--Hecke algebra $Y_{d,n}$ and a direct sum of matrix algebras over tensor products of Iwahori--Hecke algebras of type $A$. 
The direct sum is naturally indexed by the set of compositions of $n$ with $d$ parts. This result is in fact a special case of a result by G. Lusztig \cite[\S 34]{lu}.
For the subsequent developments, it is convenient for us to start from the presentation of $Y_{d,n}$ adapted to framed links theory. Therefore, for later use and in order to be self-contained, we give full details. The explicit formulas for the isomorphism (and its inverse) will allow us to concretely translate questions and problems from one side to the other. %

\vskip .1cm
Then, we first develop in Section \ref{sec-rep} the applications of the isomorphism theorem concerning representation theory. Indeed, the isomorphism can be rephrased by saying that the Yokonuma--Hecke algebra $Y_{d,n}$ is Morita equivalent to a direct sum of tensor products of Iwahori--Hecke algebras of type $A$. As the result is valid in the generic situation (over the ring $\C[u^{\pm1},v]$), it passes to the specializations of the parameters $u$ and $v$. This implies that both the semisimple and the modular representation theories of $Y_{d,n}$ can be deduced from the corresponding ones of the Iwahori--Hecke algebra of type $A$, which are well-studied (see \emph{e.g.} \cite{GP}). In particular, the classification of simple modules of $Y_{d,n}$ and the decomposition matrices (in characteristic $0$)  follow.

In addition, the isomorphism theorem provides a natural symmetrizing form on $Y_{d,n}$ derived from the canonical symmetrizing form of the Iwahori--Hecke algebra of type $A$. As a first application of the explicit formulas, we show that this symmetrizing form actually coincides with the symmetrizing form defined in \cite{CPA1}, which provides a direct proof and an explanation of the form of the Schur elements.

\paragraph{1.3.} Another class of applications of the isomorphism theorem concerns the theory of  classical and framed knots and links (Sections \ref{sec-mark} and \ref{sec-inv}).  
Indeed, we obtain a complete classification of the Markov traces on the family, on $n$, of the Yokonuma--Hecke algebras $Y_{d,n}$ (Theorem \ref{theo-mark}). This is done in two steps. First we translate, with the help of the isomorphism theorem, the Markov trace properties into properties of traces on tensor products of Iwahori--Hecke algebras of type A; then we fully characterize these traces using the known uniqueness of the Markov trace on the Iwahori--Hecke algebras of type A. In particular, we show that all the Markov traces on $Y_{d,n}$ are related with the unique Markov trace on the Iwahori--Hecke algebras of type $A$. 

\vskip .1cm
We note that we use a different definition of a Markov trace on $Y_{d,n}$ than in \cite{ju2,jula0,jula1,jula2,jula3}. In there, the standard approach initiated by Jones  for classical links was followed (see \cite{jo} and references therein). The first step is the construction on $Y_{d,n}$ of an analogue of the Ocneanu trace by Juyumaya \cite{ju2}. Additional conditions were imposed in \cite{ju2} in order to obtain the existence and unicity of this trace. Then, a rescaling procedure is necessary to construct invariants and, as it turned out, the trace does not rescale directly as in the classical case. A non-trivial rescaling procedure was implemented by Juyumaya and Lambropoulou in \cite{jula1,jula2,jula3} by means of the so-called "E-system" and led to further restrictions on the parameters.

In the definition we use here for the Markov trace on $Y_{d,n}$, the imposed conditions are the minimal ones allowing to obtain link invariants, namely, the centrality and the so-called Markov condition (see Subsection \ref{subsec-Mark}). This will allow us to avoid any kind of rescaling procedure during the construction of invariants. This approach is explained in \cite[\S 4.5]{GP} in the classical setting of the Iwahori--Hecke algebras of type A.
\vskip .1cm

With the definition used here, the space of Markov traces has a structure of $\mathbb{C}[u^{\pm 1},v^{\pm 1}]$-module. Our approach via the isomorphism theorem provides a distinguished basis, which is indexed by the set of compositions into $d$ parts with all parts equal to $0$ or $1$. Thus the space of Markov traces on the Yokonuma--Hecke algebras has dimension $2^d-1$.

\paragraph{1.4.} Finally, the last section is devoted to the construction and the study of invariants for classical links and framed links. Following Juyumaya and Lambropoulou, we realize $Y_{d,n}$ as a quotient of the framed braid group. Actually we do more: we introduce a one-parameter family of homomorphisms from the group algebra of the framed braid group to the algebra $Y_{d,n}$ (deforming the canonical homomorphism).
Then, for each homomorphism of this family and each Markov trace on $Y_{d,n}$, we construct an invariant for classical and framed links. As mentioned above, no rescaling is needed and the invariant is simply obtained by the composition of the homomorphism followed by the Markov trace.

The obtained invariants are Laurent polynomials in three variables: two of these variables are the parameters $u$ and $v$ in the definition of $Y_{d,n}$, and the third one is the parameter appearing in the homomorphism from the group algebra of the framed braid group to $Y_{d,n}$. Among these invariants, we recover the ones obtained by Juyumaya and Lambropoulou by taking a particular value for this third parameter and some specific Markov traces.

\vskip .1cm
Restricting to classical links, we use our construction via the isomorphism theorem to prove two results devoted to the comparison of the obtained invariants with the HOMFLYPT polynomial.

First we show that the HOMFLYPT polynomial is contained in them. More precisely, among the $2^d-1$ basic Markov traces, there are $d$ of them whose associated invariants coincide with the HOMFLYPT polynomial. These basic Markov traces are the ones indexed by compositions into $d$ parts with only one part equal to 1 and all the others equal to 0 (in the particular case of the Juyumaya--Lambropoulou invariants, this result corrresponds to \cite[Corollary 1]{chla}).

Then, we show that the invariants obtained from the others basic Markov traces are always equal to $0$ when applied to a classical knot. For classical knots, this solves completely the study of these invariants, which are thus shown to be topologically equivalent to the HOMFLYPT polynomial. This gives, in particular, a different proof of  results of  \cite{vous} about the  Juyumaya--Lambropoulou invariants.

\paragraph{Notations.}\begin{itemize}
\item We fix an integer $d\geq1$, and we let $\{\xi_1,\ldots ,\xi_d  \}$ be the set of roots of unity of order $d$. We will often use without mentioning that $\frac{1}{d}\sum_{0\leq s \leq d-1}\xi_a^s\xi_b^{-s}$ is equal to $1$ if $a=b$ and is equal to $0$ otherwise.
\item Let $\mathcal{A}$ be an algebra defined over a commutative ring $R$. If $R'$ is a commutative ring with a given ring homomorphism $\theta\,:\,R\to R'$, we will denote the specialized algebra $R'_{\theta}\mathcal{A}:=R'\otimes_R\mathcal{A}$ where the tensor product is defined via $\theta$.
In particular, if $R'$ is a commutative ring containing $R$ as a subring, we denote simply by $R'\mathcal{A}:=R'\otimes_R\mathcal{A}$ the algebra with ground ring extended to $R'$.
\item We will denote by $\M_{i,j}$ an elementary matrix with 1 in position $(i,j)$ and $0$ everywhere else (the size of the matrix will always be given by the context).
\end{itemize}

\paragraph{Acknowledgements.}
 The authors thank C\'edric Bonnaf\'e, Sofia Lambropoulou and Ivan Marin for useful discussions and comments. The authors are grateful to George Lusztig for indicating the reference \cite{lu}.
 The first author is supported by 
 Agence National de la Recherche Projet ACORT ANR-12-JS01-0003.

\section{Definitions and first properties}\label{sec-def}

\subsection{The Iwahori--Hecke algebra of type $A$}\label{hA}
Let $n\in \Z_{\geq 1}$ and let $u$ and $v$ be indeterminates. The {\it Iwahori--Hecke algebra} $\mathcal{H}_{n}$ of type $A_{n-1}$ is the associative  $\C[u^{\pm 1},v]$-algebra (with unit) with a presentation by generators :
$$T_1,\ldots, T_{n-1},$$
and relations :
\begin{equation}\label{def-H}
\begin{array}{rclcl}
T_iT_j & = & T_jT_i && \mbox{for all $i,j=1,\ldots,n-1$ such that $\vert i-j\vert > 1$,}\\[0.1em]
T_iT_{i+1}T_i & = & T_{i+1}T_iT_{i+1} && \mbox{for  all $i=1,\ldots,n-2$,}\\[0.1em]
T_i^2&=&u^2+v T_i &&\mbox{for  all $i=1,\ldots,n-1$.}\\[0.1em]
\end{array}
\end{equation}
Note that $\cH_1=\C[u^{\pm 1},v]$. It is convenient also to set $\cH_0:=\C[u^{\pm 1},v]$.

Let $R$ be a ring and let $\theta : \C[u^{\pm 1},v] \to R$ be a specialization such that $\theta (u^2)= 1$ and $\theta (v)=0$ then the specialized algebra $R_{\theta}\mathcal{H}_{n}$
 is naturally isomorphic to the group algebra $R\mathfrak{S}_n$ of the symmetric group.
 
 \begin{Rem}\label{RemA}  Let $q$ be an indeterminate. Another usual presentation of $\mathcal{H}_{n}$ is obtained by the specialization $\theta : \C[u^{\pm 1},v] \to \mathbb{C}[q,q^{-1}]$ given by $\theta (u^2)=1$ and $\theta (v)=q-q^{-1}$. \hfill$\triangle$
 \end{Rem}
 
 Let $w\in \mathfrak{S}_n$ and $s_{i_1} \ldots s_{i_r}$ a reduced expression of $w$ (where $(i_1,\ldots,i_r )\in \{1,\ldots,n-1\}^r$ and  $s_i\in \mathfrak{S}_n$ denotes the transposition $(i,i+1)$ for $i=1,\ldots,n-1$). Then, by Matsumoto's lemma (see \cite[\S 1.2]{GP}), the element $T_{i_1}\ldots T_{i_r}$ does not depend of the choice of the reduced expression of $w$ and thus, the element $T_w := T_{i_1}\ldots T_{i_r}$ is well-defined. Then $\mathcal{H}_{n}$  is free as a $\C[u^{\pm 1},v]$-module with basis $\{T_w \ |\ w\in \mathfrak{S}_n\}$ (see \cite[Thm. 4.4.6]{GP}). In particular it has dimension $n!$. 

We also set  $\tT_i:=u^{-1}T_i$, for $i\in\{1,\dots,n-1\}$ and,
\begin{equation}\label{def-tT}
\tT_w:=u^{-\ell(w)}T_w=\tT_{i_1}\ldots \tT_{i_r}\ ,\ \ \ \ \ \ \ \text{for $w\in\mS_n$\ .}
\end{equation}
where $\ell(w)$ is the length of $w$. The set $\{\tT_w \ |\ w\in \mathfrak{S}_n\}$ is also a $\C[u^{\pm 1},v]$-basis of $\mathcal{H}_{n}$.

\subsection{Compositions of $n$}

Let $\operatorname{Comp}_d (n)$ be the set of {\it compositions} of $n$ with $d$ parts (or $d$-compositions of $n$), that is the set of  $d$-tuples of non negative integers $\mu=(\mu_1,\ldots,\mu_d)$ such that $\sum_{1\leq a\leq d} \mu_a =n$. The set of $d$-compositions of $n$ is denoted by $\operatorname{Comp}_d(n)$. We denote by $|\mu|:=n$ the size of the  composition $\mu$.

\vskip .1cm
For $\mu \in \operatorname{Comp}_d(n)$, the Young subgroup $\mathfrak{S}^{\mu}$ is the subgroup $\mS_{\mu_1}\times\dots\times\mS_{\mu_d}$ of  $\mathfrak{S}_{n}$, where $\mS_{\mu_1}$ acts on the letters $\{1,\dots,\mu_1\}$, $\mS_{\mu_2}$ acts on the letters $\{\mu_1+1,\dots,\mu_2\}$, and so on. The subgroup $\mathfrak{S}^{\mu}$ is generated by the transposition $(i,i+1)$ with $i\in I_{\mu}:=\{1,\ldots,n-1\}\setminus \{\mu_1,\mu_1+\mu_2,\ldots, \mu_{1}+\ldots +\mu_{d-1}\}$. 

\vskip .1cm
Similarly, we have an associated subalgebra $\mathcal{H}^{\mu}$ of  $\mathcal{H}_{n}$ generated by $\{T_i \ |\ i\in I_{\mu}\}$. This is a free   $\C[u^{\pm 1},v]$-module with basis $\{T_w \ |\ w\in \mathfrak{S}^{\mu}\}$ (or $\{\tT_w \ |\ w\in \mathfrak{S}^{\mu}\}$). The algebra $\cH^{\mu}$ is naturally isomorphic to $\mathcal{H}_{\mu_1} \otimes \ldots \otimes \mathcal{H}_{\mu_r}$, where the tensor products are over $\C[u^{\pm 1},v]$. Note that the defining relations of $\mathcal{H}^{\mu}$ in terms of the generators $T_i$ with  $i\in I_{\mu}$, are the relations from (\ref{def-H}) involving only those generators.

\vskip .1cm
Let $\mu\in\Comp_d(n)$. For $a\in\{1,\dots,d\}$, we denote by $\mu^{[a]}$ the composition in $\Comp_d(n+1)$ defined by
\begin{equation}\label{def-mu-a+}
\mu^{[a]}_b:=\mu_b\ \ \text{if $b\neq a$,}\quad\ \ \ \ \text{and}\ \ \ \ \quad\mu^{[a]}_a:=\mu_a+1\ .
\end{equation} 
Conversely, if $\mu_a\geq1$, we define $\mu_{[a]}\in\Comp_d(n-1)$ to be the unique composition such  that 
\begin{equation}\label{def-mu-a-}
(\mu_{[a]})^{[a]}=\mu\ .
\end{equation}
We also define the \emph{base} of $\mu$, denoted by $[\mu]$, to be the $d$-composition defined by
 \begin{equation}\label{def-base-mu}
[\mu]_a=\left\{\begin{array}{ll}
1 & \text{if $\mu_a\geq1$,}\\[0.2em]
0 &  \text{if $\mu_a=0$,}
\end{array}\right.\ \ \ \ \ \ \text{for $a=1,\dots,d$.}
\end{equation}
The composition $[\mu]$ belongs to $\Comp_d(N)$ where $N$ is the number of non-zero parts in $\mu$. We denote by $\Comp^0_d(n)$ the set of $d$-compositions of $n$ where all the parts belong to $\{0,1\}$, and we set
\begin{equation}\label{comp0}\Comp_d^0:=\bigcup_{n\geq1}\Comp_d^0(n)=\{[\mu]\ |\ \mu\in\bigcup_{n\geq1}\Comp_d(n)\}\ .
\end{equation}

\subsection{The Yokonuma--Hecke algebra} We define the Yokonuma--Hecke algebra $Y_{d,n}$ as the associative  $\C[u^{\pm 1},v]$-algebra (with unit) with a presentation by generators:
$$g_1,g_2,\ldots,g_{n-1}, t_1,\ldots, t_n,$$
and relations:
\begin{equation}\label{rel-def-Y}\begin{array}{rclcl}
g_ig_j & = & g_jg_i && \mbox{for all $i,j=1,\ldots,n-1$ such that $\vert i-j\vert > 1$,}\\[0.1em]
g_ig_{i+1}g_i & = & g_{i+1}g_ig_{i+1} && \mbox{for  all $i=1,\ldots,n-2$,}\\[0.1em]
t_it_j & =  & t_jt_i &&  \mbox{for all $i,j=1,\ldots,n$,}\\[0.1em]
g_it_j & = & t_{s_i(j)}g_i && \mbox{for all $i=1,\ldots,n-1$ and $j=1,\ldots,n$,}\\[0.1em]
t_j^d   & =  &  1 && \mbox{for all $j=1,\ldots,n$,}\\[0.2em]
g_i^2  & = & u^2 + v  e_{i}  g_i && \mbox{for  all $i=1,\ldots,n-1$,}
\end{array}
\end{equation}
where, for all $i=1,\ldots,n-1$,
$$e_i :=\frac{1}{d}\sum\limits_{0\leq s\leq d-1}t_i^s t_{i+1}^{-s}\ .$$

Note that the elements $e_i$ are idempotents and that we have $g_ie_i=e_ig_i$ for all $i=1,\dots,n-1$.
The elements $g_i$ are invertible, with
\begin{equation}
g_i^{-1} = u^{-2} g_i - u^{-2} v  e_i\,,  \qquad \mbox{for all $i=1,\ldots,n-1$}.
\end{equation}
We also set
$$\tg_i:=u^{-1}g_i\ ,\ \ \ \ \ \ \ \ \text{for $i\in\{1,\dots,n-1\}$.}$$
We note that $\tg_i^2=1+u^{-1}ve_i\tg_i\,$ and also that $\tg_i^{-1}=\tg_i - u^{-1}v  e_i$, for $i=1,\dots,n-1$.

\vskip .1cm
Let $R$ be a ring  and let $\theta : \C[u^{\pm 1},v] \to R$ be a specialization such that $\theta (u^2)= 1$  and $\theta (v)=0$ 
 then the specialized algebra $R_{\theta} Y_{d,n}$ is naturally isomorphic to the group algebra $RG(d,1,n)$ of the complex reflection group $G(d,1,n)\cong (\mathbb{Z}/d\mathbb{Z})\wr \mathfrak{S}_n\,$. Note that in the case where $d=1$ then $Y_{1,n}$ is nothing but the Iwahori--Hecke algebra of type $A_{n-1}$ as defined in Subsection \ref{hA}. 

 \begin{Rem}\label{remY} The presentation used in \cite{CPA1} of the Yokonuma--Hecke algebra is obtained, similarly to Remark \ref{RemA}, by a specialization $\theta : \C[u^{\pm 1},v] \to \mathbb{C}[q,q^{-1}]$ such that $q$ is an indeterminate, $\theta (u^2)=1$ and $\theta (v)=q-q^{-1}$. The precise connections between the presentation above and the presentation used in \cite{chla,ju2,jula0,jula1,jula2,jula3} will be carefully investigated in Section \ref{sec-inv} (see also \cite[Remark 1]{CPA1}).
 \hfill$\triangle$
 \end{Rem}

We set, for $w\in \mathfrak{S}_n$ and $s_{i_1} \ldots s_{i_r}$ a reduced expression for $w$,
\begin{equation}\label{def-g-tg}
g_w:=g_{i_1}\ldots g_{i_r}\ \ \ \ \ \ \ \text{and}\ \ \ \ \ \ \ \tg_w:=u^{-\ell(w)}g_w=\tg_{i_1}\ldots\tg_{i_r}\ .
\end{equation}
Again, by Matsumoto's lemma (see \cite[\S 1.2]{GP}), the elements $g_w$ and $\tg_w$ are well-defined. The following multiplication rules in  $Y_{d,n}$ follow directly from the definitions. For $w\in \mathfrak{S}_n$ and $i\in \{1,\ldots,n-1\}$, we have
\begin{equation}\label{mult-r}\tg_{w}\tg_i =\left\{ \begin{array}{lcl}
  \tg_{ws_i} & \text{if} & l(ws_i)>l(w),\\[0.2em]
  \tg_{ws_i} +u^{-1}v \tg_w e_i & \text{if} & l(ws_i)<l(w);\end{array}\right.
 \end{equation}
\begin{equation}\label{mult-l}\tg_i \tg_{w} =\left\{ \begin{array}{lcl}
  \tg_{s_iw} & \text{if} & l(s_iw)>l(w),\\[0.2em]
 \tg_{s_i w} +u^{-1}ve_i \tg_w & \text{if} & l(s_i w)<l(w).\end{array}\right.
  \end{equation}
By \cite{ju2} and Remark \ref{remY}, $Y_{d,n}$  is a free $\C[u^{\pm1},v]$-module with basis 
 \begin{equation}\label{ju-basis}
 \{ t_1^{k_1} \ldots t_n^{k_n} \tg_w \ |\ w\in \mathfrak{S}_n,\ k_1,\ldots, k_n \in \mathbb{Z}/d\mathbb{Z}\}
 \end{equation}
 and the rank of  $Y_{d,n}$ is $d^n n!$. The algebra $Y_{d,n-1} $ naturally embeds in the algebra $Y_{d,n}$ in an obvious way. 

 \begin{Rem}\label{rem-u-v} For the isomorphism theorem, we will mainly use the elements $\tg_i$ and $\tg_w$ instead of $g_i$ and $g_w$. Concerning this part, we could have given the presentation of $Y_{d,n}$ in terms of the generators $\tg_i$ and thus remove one of the variables $u$ or $v$. However, the generators $g_i$ will be used systematically starting from Section \ref{sec-mark} for applications to links theory.
 \hfill$\triangle$
 \end{Rem}

 \subsection{A decomposition of $Y_{d,n}$ }
 We consider the commutative subalgebra $\mathcal{A}_n:=\langle t_1,\ldots, t_n\rangle$ of $Y_{d,n} $.
 This algebra  is naturally isomorphic to the group algebra of $(\mathbb{Z}/d\mathbb{Z})^n$ over $\C[u^{\pm1},v]$, and we will always implicitly make this identification in the following.
 
A complex character $\chi$ of the group $(\mathbb{Z}/d\mathbb{Z})^n$ is characterized by the choice of $\chi (t_j)\in \{\xi_1,\ldots ,\xi_d  \}$ for each $j=1,\ldots,n$. We denote by $\operatorname{Irr}(\mathcal{A}_n)$ the set of complex characters of $(\mathbb{Z}/d\mathbb{Z})^n$, extended to $\mathcal{A}_n$.
 \begin{Def} For each $\chi \in \operatorname{Irr} \left(\mathcal{A}_n\right)$, we denote by $E_{\chi}$ the primitive idempotent of 
 $\mathcal{A}_n$ associated to 
$\chi$, that is, characterized by $\chi'(E_{\chi})=0$ if $\chi'\neq \chi$ and $\chi(E_{\chi})=1$. 
\end{Def}
\noindent The idempotent $E_{\chi}$ is explicitly written in terms of the generators as follows:
\begin{equation}\label{E-chi}
E_{\chi} =\prod_{1\leq i\leq n}\left(\frac{1}{d}\sum_{0\leq s \leq d-1}\chi(t_i)^st_i^{-s}\right)\ .
\end{equation}
By definition, we have, for all  $\chi \in \operatorname{Irr} (\mathcal{A}_n)$ and $i=1,\ldots, n$,
\begin{equation}\label{E-chii}
t_i E_{\chi}=E_{\chi} t_i =\chi (t_i) E_{\chi}\,.
\end{equation}
The symmetric group  $\mathfrak{S}_n$ acts by permutations on $(\mathbb{Z}/d\mathbb{Z})^n$ and in turn acts on $\operatorname{Irr} (\mathcal{A}_n)$. The action is given by the formula:
 $$w(\chi)\bigl(t_i\bigr)=\chi (t_{w^{-1} (i)})\,,\ \ \ \ \ \ \ \ \text{for all $i=1,\ldots,n$,  $w \in \mathfrak{S}_n$ and $\chi\in\operatorname{Irr} (\mathcal{A}_n)$.}$$
In the algebra $Y_{d,n}$, due to the relation $g_wt_i=t_{w(i)}g_w$ for $i=1,\dots,n$ and $w\in \mathfrak{S}_n$, we have
\begin{equation}\label{rel-g-E}
g_w E_{\chi}=E_{w(\chi)} g_{w}\ \ \ \ \ \quad\text{and}\quad\ \ \ \ \ \tg_w E_{\chi}=E_{w(\chi)} \tg_{w}\ .
\end{equation}

\vskip .1cm
Let $\chi\in\operatorname{Irr} (\mathcal{A}_n)$. For $a=1,\ldots,d$, denote by $\mu_a$ the cardinal of the set $\{ j\in \{1,\ldots,n\} \ |\ \chi (t_j)=\xi_a\}$. Then the sequence $(\mu_1,\ldots,\mu_d)$ is a $d$-composition of $n$ which is denoted by 
$$\operatorname{Comp}(\chi):=(\mu_1,\ldots,\mu_d)\in\Comp_d(n)\ .$$ 
Let $\mu\in \operatorname{Comp}_d (n)$. Then we denote by 
$$\mathcal{O} (\mu):=\{ \chi \in  \operatorname{Irr} (\mathcal{A}_n)\ |\ \operatorname{Comp} (\chi)=\mu\}$$
the orbit of the element $\chi \in  \operatorname{Irr} (\mathcal{A}_n)$ under the action of the symmetric group and 
$$m_{\mu}:=\sharp \mathcal{O}_{\mu}=\displaystyle \frac{n!}{\mu_1!\mu_2!\dots\mu_d!}\ .$$

\begin{Def}
Let $\mu\in \operatorname{Comp}_d (n)$. We set
$$E_{\mu}:=\sum_{\operatorname{Comp}(\chi)=\mu} E_{\chi}=  \sum_{\chi \in \mathcal{O} (\mu)} E_{\chi}\ .$$
\end{Def}
Due to the commutation relation (\ref{rel-g-E}), the elements $E_{\mu}$, with $\mu\in \operatorname{Comp}_d (n)$,  are central in $Y_{d,n}$. Moreover, as the set $\{E_{\chi}\ |\ \chi\in\operatorname{Irr} (\mathcal{A}_n)\}$ is a complete set of orthogonal idempotents, it follows at once that the set $\{ E_{\mu} \ |\ \mu\in \operatorname{Comp}_d (n)\}$ forms a complete set of central orthogonal idempotents in $Y_{d,n}$. In particular, we have the following decomposition of $Y_{d,n}$ into a direct sum of two-sided ideals:
\begin{equation}\label{Emu-Y}
Y_{d,n} =\bigoplus_{\mu \in \operatorname{Comp}_d (n)} E_{\mu} Y_{d,n}\ .
\end{equation}

 \subsection{Another basis for $Y_{d,n} $ }\label{subsec-basis}
 We here give another basis for $Y_{d,n}$ using the idempotents we just defined. As the subalgebra $\mathcal{A}_n$ of $Y_{d,n}$ is isomorphic to the group algebra of $(\mathbb{Z}/d\mathbb{Z})^n$ over $\C[u^{\pm1},v]$,   the set $\{ E_{\chi}  \ |\ \chi \in  \operatorname{Irr} (\mathcal{A}_n)\}$ is a $\C[u^{\pm1},v]$-basis of $\mathcal{A}_n$, as well as the set $\{ t_1^{k_1} \ldots t_n^{k_n} \ |\ k_1,\ldots, k_n \in \mathbb{Z}/d\mathbb{Z}\}$. So from the knowledge of the $\C[u^{\pm1},v]$-basis (\ref{ju-basis}) of $Y_{d,n}$, we also have that the set
\begin{equation}\label{E-basis}
\{ E_{\chi} \tg_w \ |\ \chi \in  \operatorname{Irr} (\mathcal{A}_n),\ w\in \mathfrak{S}_n\}
\end{equation}
 is a $\C[u^{\pm1},v]$-basis of $Y_{d,n}$. Moreover, this basis is compatible with the decomposition (\ref{Emu-Y}) of $Y_{d,n}$ since, for $\mu \in \operatorname{Comp}_d (n)$, we have $E_{\chi} \tg_w\in E_{\mu} Y_{d,n}$ if and only if $\operatorname{Comp}(\chi)=\mu$. In other words, the set
 \[\{ E_{\chi} \tg_w \ |\ \chi \in  \operatorname{Irr} (\mathcal{A}_n)\ \text{with $\operatorname{Comp}(\chi)=\mu$},\ \ w\in \mathfrak{S}_n\}\]
 is a $\C[u^{\pm1},v]$-basis of $E_{\mu}Y_{d,n}$.
 
 \vskip .1cm
Now we will label the elements of  $\operatorname{Irr} (\mathcal{A}_n)$ in a useful way for the following. This is done as follows. We first consider a distinguished element in each orbit $\mathcal{O} (\mu)$. Let $\mu\in\Comp_d(n)$. We denote
$$\chi_1^{\mu}\in \operatorname{Irr} (\mathcal{A}_n)\ ,$$
the character given by
\begin{equation}\label{chi1-mu}
\left\{\begin{array}{ccccccc}
\chi_1^{\mu} (t_1)&=&\ldots & =& \chi_1^{\mu} (t_{\mu_1})&=& \xi_1\ ,\\[0.2em]
\chi_1^{\mu} (t_{\mu_1+1})&=&\ldots & =& \chi_1^{\mu} (t_{\mu_1+\mu_2})&=& \xi_2\ ,\\
\vdots &\vdots &\vdots &\vdots &\vdots &\vdots &\vdots  \\
\chi_1^{\mu} (t_{\mu_1+\dots+\mu_{d-1}+1})&=&\ldots & =& \chi_1^{\mu} (t_{\mu_d})&=& \xi_d\ .\\
\end{array}\right.
\end{equation}
Note that the stabilizer of $\chi_1^{\mu}$ under the action of $\mathfrak{S}_n$ is the Young subgroup $\mathfrak{S}^{\mu}$. In each left coset in $\mathfrak{S}_n/\mathfrak{S}^{\mu}$, we take a representative of minimal length (such a representative is unique, see \cite[\S 2.1]{GP}). We denote by 
$$\{\pi_{1,{\mu}},\ldots,\pi_{m_{\mu},{\mu}}\}$$
this set of distinguished left coset representatives of $\mathfrak{S}_n/\mathfrak{S}^{\mu}$ with the convention that $\pi_{1,{\mu}}=1$ (recall that $m_{\mu}:=\sharp \mathcal{O} (\mu)$). Then, if we set for all 
 $k=1,\ldots, m_{\mu}$:
 \begin{equation}\label{defc}
 \chi_k^{\mu}:=\pi_{k,{\mu}} (\chi_1^{\mu}),
 \end{equation}
 we have by construction that
$$\mathcal{O} (\mu)=\{\chi_1^{\mu}, \ldots,\chi_{m_{\mu}}^{\mu}\}\ .$$

To sum up, we have the following $\C[u^{\pm1},v]$-basis of $Y_{d,n}$:
\begin{equation}\label{Ekmu-basis}
\{ E_{\chi^\mu_{k}}\tg_w\ |\ w\in \mathfrak{S}_n,\ k=1,\ldots, m_{\mu},\ \mu\in\Comp_d(n)\}\ ,
\end{equation}
where, for each $\mu\in\Comp_d(n)$, the subset $\{ E_{\chi^\mu_{k}}\tg_w\ |\ w\in \mathfrak{S}_n,\ k=1,\ldots, m_{\mu}\}$ is a $\C[u^{\pm1},v]$-basis of the two-sided ideal $E_{\mu}Y_{d,n}$.

 \section{The isomorphism theorem}\label{sec-iso}
 
The aim of this part is to study an isomorphism between $Y_{d,n}$ and $\bigoplus_{\mu \in \operatorname{Comp}_d (n)} \operatorname{Mat}_{m_{\mu}} (\mathcal{H}^{\mu}) $, and then to analyse the inclusion $Y_{d,n-1}\subset Y_{d,n}$ from this point of view. Theorem \ref{main} below is in fact a special case of a result by G. Lusztig (see \cite[\S 34]{lu}). Nevertheless we give full details of the proof as the explicit formulas will then be used throughout the rest of the paper.  %
 
 \subsection{The statement}
 Let $\mu\in \operatorname{Comp}_d (n)$. We recall that $E_{\mu} Y_{d,n}$ is a two-sided ideal of $Y_{d,n}$ and is also a unital subalgebra with unit $E_{\mu}$. We define a linear map 
 $$\Phi_{\mu}: \operatorname{Mat}_{m_{\mu}} (\mathcal{H}^{\mu}) \to E_{\mu} Y_{d,n}\ ,$$
by setting, for any matrix consisting of basis elements $\tT_{w_{i,j}}$ of $\cH^{\mu}$ (that is, with $w_{i,j}\in\mS^{\mu}$),
\begin{equation}\label{def-phi}
\Phi_{\mu}\bigl((\tT_{w_{i,j}})_{1\leq i,j\leq m_{\mu}}\bigr)=\sum_{1\leq i,j\leq m_{\mu}}   E_{\chi^{\mu}_i} \,\tg_{\pi_{i,\mu} w_{i,j} \pi_{j,\mu}^{-1}}\, E_{\chi^{\mu}_j}\ .
\end{equation}
We also define a linear map 
 $$\Psi_{\mu}: E_{\mu} Y_{d,n}  \to \operatorname{Mat}_{m_{\mu}} (\mathcal{H}^{\mu} ) \ ,$$
as follows. Let $k\in\{1,\dots,m_{\mu}\}$ and $w\in\mS_n$, and let $j\in\{1,\dots,m_{\mu}\}$ be uniquely defined (given $k$) by the relation $w(\chi^{\mu}_j)=\chi^{\mu}_k$. Note that we thus have ${\pi_{k,{\mu}}^{-1} w \pi_{j,{\mu}}}\in \mathfrak{S}^{\mu}$. We then set
\begin{equation}\label{def-psi}
\Psi_{\mu} (E_{\chi^{\mu}_k} \tg_w)=\tT_{\pi_{k,{\mu}}^{-1} w \pi_{j,{\mu}}}\,\M_{k,j}\ ,
\end{equation}
where we recall that $\M_{k,j}$ denotes the elementary matrix with 1 in position $(k,j)$.

\vskip .1cm
Now we can state the isomorphism theorem. Recall the decomposition (\ref{Emu-Y}) of $Y_{d,n}$. %
 \begin{Th}\label{main}  Let $\mu\in \operatorname{Comp}_d (n)$. The linear map 
$\Phi_{\mu}$ is an isomorphism of $\C[u^{\pm1},v]$-algebra with inverse map $\Psi_{\mu}$. In turn,
$$\Phi_n:=\bigoplus_{\mu \in \operatorname{Comp}_d (n) } \Phi_{\mu}\ :  \bigoplus_{\mu \in \operatorname{Comp}_d (n) } \operatorname{Mat}_{m_{\mu}} (\mathcal{H}^{\mu} )  \to Y_{d,n} $$
is an isomorphism with inverse map:
$$\Psi_n:=\bigoplus_{\mu \in \operatorname{Comp}_d (n) } \Psi_{\mu}\ :\ Y_{d,n} \to \bigoplus_{\mu \in \operatorname{Comp}_d (n) }  \operatorname{Mat}_{m_{\mu}}   (\mathcal{H}^{\mu}).   $$
 \end{Th}

 \subsection{Preliminary results}
We first prove a series of useful lemmas.
\begin{lemma}\label{idem}
Let   $\mu\in \operatorname{Comp}_d (n)$ and $i\in \{1,\ldots, m_{\mu}\}$. We consider a reduced expression 
 $s_{i_1}\ldots s_{i_k}$ of $\pi_{i,{\mu}}$.  Then for all $l\in \{1,\ldots,k\}$, we have:
 $$e_{i_l} E_{s_{i_{l+1}} \ldots s_{i_k} (\chi_1^{\mu})}=E_{s_{i_{l+1}} \ldots s_{i_k} (\chi_1^{\mu})}e_{i_l}=0$$
\end{lemma}
\begin{proof}
By definition, $\pi_{i,{\mu}}$ is the (unique) element of $\mathfrak{S}_n$ with minimal length satisfying $\pi_{i,{\mu}} (\chi_1^{\mu})=\chi_i^{\mu}$. As a consequence, we have 
  for all $l=1,\ldots, k$:
  $$s_{i_l}\cdot s_{i_{l+1}} \ldots s_{i_k} (\chi_1^{\mu})\neq s_{i_{l+1}} \ldots s_{i_k} (\chi_1^{\mu})$$
which is equivalent to
  $$s_{i_{l+1}} \ldots s_{i_k} (\chi_1^{\mu}) (t_{i_l})\neq    s_{i_{l+1}} \ldots s_{i_k} (\chi_1^{\mu}) (t_{i_{l}+1})\ .$$
Thus by (\ref{E-chii}), we have 
  $$t_{i_l} E_{s_{i_{l+1}} \ldots s_{i_k} (\chi_1^{\mu})}\neq t_{i_l+1} E_{s_{i_{l+1}} \ldots s_{i_k} (\chi_1^{\mu})}.$$
This discussion shows that 
    $$t_{i_l} t^{-1}_{i_{l+1}}E_{s_{i_{l+1}} \ldots s_{i_k} (\chi_1^{\mu})}= \xi_j E_{s_{i_{l+1}} \ldots s_{i_k} (\chi_1^{\mu})}$$
    for a $d$-root of unity $\xi_j\neq 1$. We conclude that
   $$E_{s_{i_{l+1}} \ldots s_{i_k} (\chi_1^{\mu})}e_{i_l}=e_{i_l} E_{s_{i_{l+1}} \ldots s_{i_k} (\chi_1^{\mu})}=  (\sum_{0\leq s\leq d-1} \xi_j^s) E_{s_{i_{l+1}} \ldots s_{i_k} (\chi_1^{\mu})}=0\ ,$$
where we note, for the first equality, that $e_{i_l}$ commutes with any $E_{\chi}$.
\end{proof}

 \begin{lemma}\label{l1}
 For all  $\mu\in \operatorname{Comp}_d (n)$, $1\leq i,j\leq m_{\mu}$ and $w\in\mS_n$, we have:
 \begin{enumerate}
 \item[(i)] $E_{\chi_1^{\mu}}\ \tg_{\pi_{i,\mu}}^{-1} \tg_w \tg_{\pi_{j,\mu}} E_{\chi_1^{\mu}}=E_{\chi_1^{\mu}}\,\tg_{\pi_{i,\mu}^{-1}w \pi_{j,\mu}} E_{\chi_1^{\mu}}$\ ;
 \item[(ii)] $E_{\chi_i^{\mu}}\,\tg_{\pi_{i,\mu}} \tg_w \tg^{-1}_{\pi_{j,\mu}} E_{\chi_j^{\mu}}=E_{\chi_i^{\mu}} \, \tg_{\pi_{i,\mu} w \pi_{j,\mu}^{-1} } E_{\chi_j^{\mu}}\ .$
 \end{enumerate}
  \end{lemma} 
 \begin{proof}
 Let us denote a reduced expression of $\pi_{i,\mu}$ by $s_{i_1}\ldots s_{i_k}$. 
 We have :
  $$\begin{array}{rcl}
  E_{\chi_1^{\mu}}\ \tg_{\pi_{i,\mu}}^{-1}\ \tg_w &= &   E_{\chi_1^{\mu}}\,\tg^{-1}_{{i_k}} \ldots \tg_{{i_1}}^{-1}\  \tg_w\\[0.6em]
  &=& \tg^{-1}_{{i_k}} \ldots  \tg_{{i_2}}^{-1}\ E_{s_{i_2} \ldots s_{i_k}(\chi_1^{\mu})}\ \tg_{{i_1}}^{-1}\ \tg_w\ .
  \end{array}$$
  Recall that for all $j=1,\ldots,n$, we have $\tg_j^{-1}= \tg_j -u^{-1}v e_j$. Thus, with repeated applications of Lemma \ref{idem} together with the multiplication rule (\ref{mult-l}), we deduce that:
    $$\begin{array}{rcl}
E_{\chi_1^{\mu}}\ \tg_{\pi_{i,\mu}}^{-1}\ \tg_w  &= &      \tg^{-1}_{{i_k}} \ldots  \tg_{{i_2}}^{-1}\  E_{s_{i_2} \ldots s_{i_k}(\chi_1^{\mu})}\ \tg_{s_{i_1}w}         \\[0.6em]
&= &      \tg^{-1}_{{i_k}} \ldots  \tg_{{i_3}}^{-1}\ E_{s_{i_3} \ldots s_{i_k}(\chi_1^{\mu})}\ \tg_{{i_2}}^{-1}\ \tg_{s_{i_1}w}         \\[0.5em]
  &=& \ldots\\[0.3em]
  &=&  E_{{\chi_1^{\mu}}}\ \tg_{\pi_{i,\mu}^{-1}w}\ .
  \end{array}$$
Now let us denote by $s_{j_1}\ldots s_{j_l}$ a reduced expression of  $\pi_{j,\mu}$, we have:
    $$\begin{array}{rcl}
E_{\chi_1^{\mu}}\ \tg_{\pi_{i,\mu}}^{-1}\ \tg_w\ \tg_{\pi_{j,\mu}}\ E_{\chi_1^{\mu}}&=&E_{{\chi_1^{\mu}}} \ \tg_{\pi_{i,{\mu}}^{-1}w}\ \tg_{\pi_{j,\mu}}\ E_{\chi_1^{\mu}}\\[0.6em]
&=& E_{{\chi_1^{\mu}}}\ \tg_{\pi_{i,{\mu}}^{-1}w}\ \tg_{{j_1}}\ldots \tg_{{j_l}}\ E_{\chi_1^{\mu}}\\[0.6em]
&=&E_{{\chi_1^{\mu}}}\ \tg_{\pi_{i,\mu}^{-1}w}\ \tg_{{j_1}}\ E_{s_{j_2} \ldots s_{j_l}(\chi_1^{\mu})}\ \tg_{i_2}\ldots \tg_{i_l}\ .
  \end{array}$$
As above, with repeated applications of Lemma \ref{idem} together with the multiplication rule (\ref{mult-r}), we obtain:
   $$\begin{array}{rcl}
E_{\chi_1^{\mu}}\ \tg_{\pi_{i,\mu}}^{-1}\ \tg_w\ \tg_{\pi_{j,\mu}}\ E_{\chi_1^{\mu}}&=&  E_{{\chi_1^{\mu}}}\  \tg_{\pi_{i,{\mu}}^{-1}ws_{j_1}}\ E_{s_{j_2} \ldots s_{j_l}(\chi_1^{\mu})}\ \tg_{j_2}\ldots \tg_{j_l}\\[0.6em]
&=&  E_{{\chi_1^{\mu}}}\ \tg_{\pi_{i,\mu}^{-1}ws_{j_1}}\ \tg_{j_2}\ E_{s_{j_3} \ldots s_{j_l}(\chi_1^{\mu})}\ \tg_{j_3}\ldots \tg_{j_l}\\[0.5em]
&=&\ldots\\[0.3em]
&=&E_{\chi_1^{\mu}}\ \tg_{\pi_{i,\mu}^{-1}w \pi_{j,\mu}}\ E_{\chi_1^{\mu}}\ ,
  \end{array}$$
 which  proves item $(i)$. Let us now prove item $(ii)$.  We have by
 (\ref{defc}) and 
 (\ref{rel-g-E}) :
    $$\begin{array}{rcl}
E_{\chi_i^{\mu}}\ \tg_{\pi_{i,\mu}}\ \tg_w\ \tg^{-1}_{\pi_{j,\mu}}\ E_{\chi_j^{\mu}}&=& \ \tg_{\pi_{i,\mu}}\,  E_{\chi_1^{\mu}}\ \tg_w \ E_{\chi_1^{\mu}}\ \tg^{-1}_{\pi_{j,\mu}} \\[0.6em]
&=&\tg_{\pi_{i,\mu}}\,E_{\chi_1^{\mu}}\ \tg_{\pi_{i,\mu}^{-1}\pi_{i,\mu} w \pi_{j,\mu}^{-1} \pi_{j,\mu}}\  E_{\chi_1^{\mu}}\ \tg^{-1}_{\pi_{j,\mu}}\\[0.6em]
&=&\tg_{\pi_{i,\mu}}\,E_{\chi_1^{\mu}}\ \tg_{\pi_{i,\mu}}^{-1}\ \tg_{\pi_{i,\mu} w \pi_{j,\mu}^{-1}}\ \tg_{ \pi_{j,\mu}}\ E_{\chi_1^{\mu}}\ \tg^{-1}_{\pi_{j,\mu}}\ ,
  \end{array}$$
where the last equality comes from item $(i)$. The proof is concluded using that $\tg_{\pi_{i,\mu}}E_{\chi_1^{\mu}}\,\tg_{\pi_{i,\mu}}^{-1}=E_{\chi_i^{\mu}}$ and $\tg_{ \pi_{j,\mu}}E_{\chi_1^{\mu}}\,\tg^{-1}_{\pi_{j,\mu}}=E_{\chi_j^{\mu}}$.
 \end{proof}
 
  \begin{lemma}\label{l2}
 Let $\mu\in \operatorname{Comp}_d (n)$. The map 
$$\phi_{\mu} :\mathcal{H}^{\mu} \to E_{\chi_1^{\mu}} Y_{d,n} E_{\chi_1^{\mu}},$$
defined on the generators by 
$$\forall i\in I_{\mu},\ \phi_{\mu} (T_i)= E_{\chi_1^{\mu}} g_i E_{\chi_1^{\mu}},$$
extends to an homomorphism of algebras. 

 \end{lemma}
 \begin{proof} Recall that the subspace $E_{\chi_1^{\mu}} Y_{d,n} E_{\chi_1^{\mu}}$ is a unital subalgebra of $Y_{d,n}$ with unit $E_{\chi_1^{\mu}}$.
 
 We first note that if $i\in I_{\mu}$ then $s_i(\chi_{1}^{\mu})=\chi_1^{\mu}$. We thus have 
 $$g_i E_{ \chi_{1}^{\mu}}=E_{ \chi_{1}^{\mu}} g_i$$
 and this relation easily  implies that  the elements $E_{\chi_1^{\mu}} g_i E_{\chi_1^{\mu}}$ with 
 $i\in I_{\mu}$ satisfy the braid relations. It remains to check the ``quadratic relation''. We have 
 $$
 \begin{array}{rcl}
 (E_{\chi_1^{\mu}} g_i E_{\chi_1^{\mu}})^2&=&E_{\chi_1^{\mu}} g_i^2 E_{\chi_1^{\mu}}\\
 &=&u^2 E_{\chi_1^{\mu}} + v E_{\chi_1^{\mu}} e_i g_i E_{\chi_1^{\mu}}\\
  &=&u^2 E_{\chi_1^{\mu}} + v E_{\chi_1^{\mu}} g_i E_{\chi_1^{\mu}}
  \end{array}
 $$
The last equality comes from the fact that for $i\in I_{\mu}$ we have  $t_i E_{ \chi_{1}^{\mu}}=t_{i+1} E_{ \chi_{1}^{\mu}},$ and thus 
$$e_i E_{ \chi_{1}^{\mu}}=E_{ \chi_{1}^{\mu}}e_i=E_{ \chi_{1}^{\mu}}.$$
Thus all the defining relations of $\cH^{\mu}=\cH_{\mu_1}\otimes\dots\otimes\cH_{\mu_d}$ are satisfied so that
  $\phi_{\mu}$ can be extended to a homomorphism of algebras. 
 \end{proof}
 
 \begin{Rem}
One can actually show that the morphism $\phi_{\mu}$ is an isomorphism. Indeed the lemma implies that $\phi_{\mu}$ is given on the standard basis of $\cH^{\mu}$ by $\phi_{\mu}(T_w)=E_{\chi_1^{\mu}} g_w E_{\chi_1^{\mu}}$, $w\in\mS^{\mu}$. Moreover, if $w\in\mS^{\mu}$ then $w(\chi_1^{\mu})=\chi_1^{\mu}$ and therefore $\phi_{\mu}(T_w)=E_{\chi_1^{\mu}} g_w$. So it remains to check that $\{ E_{\chi_1^{\mu}} g_w \ |\ w\in \mathfrak{S}^{\mu}\}$ is a basis of $E_{\chi_1^{\mu}} Y_{d,n} E_{\chi_1^{\mu}}$. The linear independence is immediate, while the spanning property follows from the following calculation, for a basis element $E_{\chi_i^{\nu}}g_w$ of $Y_{d,n}$ 
 and $\nu \in \operatorname{Comp}_d (n)$,
\[E_{\chi_1^{\mu}}\cdot E_{\chi_i^{\nu}}g_w\cdot E_{\chi_1^{\mu}}=E_{\chi_1^{\mu}}E_{\chi_i^{\nu}} E_{w(\chi_1^{\mu})}g_w=\left\{\begin{array}{ll}E_{\chi_1^{\mu}}g_w & \text{if $\mu=\nu$, $i=1$ and $w\in\mS^{\mu}$;}\\[0.3em]
0 & \text{otherwise.}\end{array}\right.\]
In the following, we will not use the fact that $\phi_{\mu}$ is actually an isomorphism (actually, it is a consequence of Theorem \ref{main}).
 \hfill$\triangle$ 
 \end{Rem}

 \subsection{Proof of Theorem \ref{main}} 
Let $\mu\in \operatorname{Comp}_d (n)$.

\vskip .1cm
\textbf{1.} We first show that $\Phi_{\mu}$ is a morphism. Before this, we note that by Lemma \ref{l1}(ii), for all 
 $1\leq i,j \leq m_{\mu}$ and $w\in\mS^{\mu}$, we have:
\begin{equation}\label{Phi-phi}
\begin{array}{rcl}
 E_{\chi^{\mu}_i}\ \tg_{\pi_{i,\mu} w \pi_{j,\mu}^{-1}}\ E_{\chi^{\mu}_j} &=& 
  E_{\chi^{\mu}_i}\ \tg_{\pi_{i,\mu}}\ \tg_{w}\ \tg_{\pi_{j,\mu}}^{-1}\ E_{\chi^{\mu}_j}\\[0.5em]
  &=& 
  \tg_{\pi_{i,\mu}}\ E_{\chi^{\mu}_1}\ \tg_{w}\,E_{\chi^{\mu}_1}\ \tg_{\pi_{j,\mu}}^{-1}\\[0.5em]
 &=&  \tg_{\pi_{i,\mu}}\ \phi_{\mu}(\tT_{w})\ \tg_{\pi_{j,\mu}}^{-1}\ .
\end{array}
\end{equation}

Now, let $i,j,k,l\in\{1,\dots,m_{\mu}\}$ and $w,w'\in\mS^{\mu}$. We have:
\[\Phi_{\mu}\bigl( \tT_w\,\M_{i,j}\bigr)\ \Phi_{\mu}\bigl( \tT_{w'}\,\M_{k,l}\bigr) = 
E_{\chi^{\mu}_i}\ \tg_{\pi_{i,\mu} w \pi_{j,\mu}^{-1}}\ E_{\chi^{\mu}_j}
\  E_{\chi^{\mu}_k}\ \tg_{\pi_{k,\mu} w' \pi_{l,\mu}^{-1}}\ E_{\chi^{\mu}_l}\ .\]
As $E_{\chi^{\mu}_j}$ and $E_{\chi^{\mu}_k}$ belong to a family of pairwise orthogonal idempotents, this is equal to 0 if $j\neq k$. On the other hand, we also have that $\tT_w\,\M_{i,j}\cdot \tT_{w'}\,\M_{k,l}$ is equal to 0 if $j\neq k$. 

So it remains only to consider the situation $j=k$. If $j=k$, we obtain
\begin{equation}\label{Phi1}\begin{array}{rcl}
\Phi_{\mu}\bigl( \tT_w\,\M_{i,j}\bigr)\ \Phi_{\mu}\bigl( \tT_{w'}\,\M_{j,l}\bigr) & = & 
 E_{\chi^{\mu}_i}\ \tg_{\pi_{i,\mu} w \pi_{j,\mu}^{-1}}\ E_{\chi^{\mu}_j}
\,\cdot\, E_{\chi^{\mu}_j}\ \tg_{\pi_{j,\mu} w' \pi_{l,\mu}^{-1}}\ E_{\chi^{\mu}_l} \\[0.6em]
 &=& 
 \tg_{\pi_{i,\mu}}\ \phi_{\mu}(\tT_{w})\ \tg_{\pi_{j,\mu}}^{-1}\,\cdot\,\tg_{\pi_{j,\mu}}\ \phi_{\mu}(\tT_{w'})\ \tg_{\pi_{l,\mu}}^{-1} \\[0.6em]
 &=& 
 \tg_{\pi_{i,\mu}}\ \phi_{\mu}(\tT_{w}\cdot\tT_{w'})\ \tg_{\pi_{l,\mu}}^{-1}\ ,
\end{array}
\end{equation}
where we used successively (\ref{Phi-phi}) and Lemma \ref{l2}. On the other hand, we have that $\tT_w\,\M_{i,j}\cdot \tT_{w'}\,\M_{j,l}$ is equal to $\tT_w\,\tT_{w'}\,\M_{i,l}$. The product $\tT_w\,\tT_{w'}$ can be written uniquely as
\[\tT_w\,\tT_{w'}=\sum_{x\in\mS^{\mu}}c_x\tT_x\ ,\]
for some coefficients $c_x\in\C[u^{\pm1},v]$. We have now
$$\begin{array}{rcl}
\Phi_{\mu}\bigl(\tT_w\,\tT_{w'}\,\M_{i,l}\bigr) & = & 
\displaystyle\sum_{x\in\mS^{\mu}}c_x\ E_{\chi^{\mu}_i}\ \tg_{\pi_{i,\mu} x \pi_{j,\mu}^{-1}}\ E_{\chi^{\mu}_j}\\[0.6em]
 &=& 
\displaystyle \sum_{x\in\mS^{\mu}}c_x\ \tg_{\pi_{i,\mu}}\ \phi_{\mu}(\tT_x)\ \tg_{\pi_{l,\mu}}^{-1}\ ,
\end{array}
$$
by (\ref{Phi-phi}) again. Comparing with (\ref{Phi1}) concludes the verification that $\Phi_{\mu}$ is a morphism of algebras.

\vskip .1cm
\textbf{2.} We now check that $\Phi_{\mu}$ and $\Psi_{\mu}$ 
 are inverse maps.  Let $ w\in \mathfrak{S}_n$ and let $i\in\{1,\dots,m_{\mu}\}$. Let also $j\in\{1,\dots,m_{\mu}\}$ be uniquely defined by $\chi_i^{\mu}=w(\chi_j^{\mu})$. By definition of $\Phi_{\mu}$ and $\Psi_{\mu}$, we have
$$\Phi_{\mu} \circ \Psi_{\mu} (E_{\chi^{\mu}_i}\,\tg_{w})\ =\ \Phi_{\mu}(\tT_{\pi_{i,\mu}^{-1}w\pi_{j,\mu}}\,\M_{i,j})\ =\ E_{\chi_i^{\mu}}\,\tg_w\,E_{\chi_j^{\mu}}\ .$$
As $\chi_j^{\mu}=w^{-1}(\chi_i^{\mu})$ and $E_{\chi_i^{\mu}}$ is an idempotent, we conclude that this is indeed equal to $E_{\chi^{\mu}_i}\,\tg_{w}$.

\vskip .1cm
On the other hand, let $w\in\mS^{\mu}$ and $i,j\in\{1,\dots,m_{\mu}\}$.
$$\begin{array}{rcll}
\Psi_{\mu} \circ \Phi_{\mu} \bigl(\tT_w\,\M_{i,j}\bigr)    &=& \Psi_{\mu}\bigl(    
E_{\chi^{\mu}_i}\ \tg_{\pi_{i,\mu}w \pi_{j,\mu}^{-1}}\ E_{\chi^{\mu}_j}\bigr) & \\[0.6em]
&=& \Psi_{\mu}\bigl(E_{\chi^{\mu}_i}\ \tg_{\pi_{i,\mu} w \pi_{j,\mu}^{-1}}\bigr) & (\,\text{because $\pi_{i,\mu} w_{i,j} \pi_{j,\mu}^{-1} (\chi_j^{\mu})=\chi_i^{\mu}$}\,) \\[0.6em]
&=& \tT_{  \pi_{i,\mu}^{-1}  \pi_{i,\mu} w \pi_{j,\mu}^{-1}\pi_{j',\mu}}\,\M_{i,j'}\ , &
\end{array}$$
where the integer $j'\in\{1,\dots,m_{\mu}\}$ is uniquely defined by $\pi_{i,\mu} w \pi_{j,\mu}^{-1} (\chi_{j'}^{\mu})=\chi_{i}^{\mu}$. As $w\in\mS^{\mu}$, this condition yields $j'=j$, which concludes the proof.\hfill$\square$

\begin{Exa}\label{exa1}
 Let $d=2$ and $n=4$. We will give explicitly in this example the images of $g_1,\dots,g_{n-1}$, $t_1,\dots,t_n$ and $e_1,\dots,e_{n-1}$ of $Y_{d,n}$ under the isomorphism $\Psi_n$ of Theorem \ref{main}. In the matrices below, the dots stand for coefficients equal to 0.

\vskip .1cm
First, we note that, for any $\mu\in\Comp_d(n)$, the matrix $\Psi_{\mu}(t_i)$ ($i\in\{1,\dots,n\}$) is diagonal, more precisely, we have:
\[\Psi_{\mu}(t_i)=\Psi_{\mu}(E_{\mu}t_i)=\Psi_{\mu}(\sum_{1\leq k\leq m_{\mu}}E_{\chi_k^{\mu}}t_i)=\Psi_{\mu}(\sum_{1\leq k\leq m_{\mu}}\chi_k^{\mu}(t_i)E_{\chi_k^{\mu}})=
\sum_{1\leq k\leq m_{\mu}}\chi_k^{\mu}(t_i)\,\M_{k,k}\ .\]
We will denote by $\text{Diag}(x_1,\dots,x_N)$ a diagonal matrix with coefficients $x_1,\dots,x_N$ on the diagonal. We also recall that, for $i=1,\dots,n-1$, we have $g_i=u\tg_i$ and
\[\Psi_{\mu}(\tg_i)=\Psi_{\mu}(E_{\mu}\tg_i)=\Psi_{\mu}(\sum_{1\leq k\leq m_{\mu}}E_{\chi_k^{\mu}}\tg_i)=
\sum_{1\leq k\leq m_{\mu}}\tT_{\pi_{k,{\mu}}^{-1} s_i \pi_{j_k,{\mu}}}\,\M_{k,j_k}\ ,\]
where, for each $k\in\{1,\dots,m_{\mu}\}$, the integer $j_k\in\{1,\dots,m_{\mu}\}$ is uniquely determined by $s_i(\chi^{\mu}_{j_k})=\chi^{\mu}_k$.

\vskip .2cm
\noindent $\bullet$ Let $\mu=(4,0)$ or $\mu=(0,4)$. Then $m_{\mu}=1$ and $\cH^{\mu}\cong\cH_4$. There is only one character in the orbit $\mathcal{O} (\mu)$, which is $\chi^{\mu}_1=(\xi_a,\xi_a,\xi_a,\xi_a)$, where $a=1$ if $\mu=(4,0)$ and $a=2$ if $\mu=(0,4)$. In this situation, we have
\[g_i\mapsto (T_i)\,,\ \ \ \ \ t_j\mapsto(\xi_a)\,,\ \ \ \ \ e_i\mapsto(1)\,,\ \ \ \ \ \ \ \text{for $i=1,2,3\,$ and $j=1,2,3,4\,$,}\]
where $a=1$ if $\mu=(4,0)$ and $a=2$ if $\mu=(0,4)$. 

\vskip .2cm
\noindent $\bullet$
 Let $\mu=(3,1)$. Then $m_{\mu}=4$ and $\cH^{\mu}\cong\cH_3\otimes\cH_1$ and we identify it below with $\cH_3$. We order the characters in the orbit $\mathcal{O} (\mu)$ as follows: 
$$\chi^{\mu}_1=(\xi_1,\xi_1,\xi_1,\xi_2)\,,\ \ \ \ \chi^{\mu}_2=(\xi_1,\xi_1,\xi_2,\xi_1)\,,\ \ \ \ \chi^{\mu}_3=(\xi_1,\xi_2,\xi_1,\xi_1)\ \ \ \text{and}\ \ \ \chi^{\mu}_4=(\xi_2,\xi_1,\xi_1,\xi_1)\ .$$
Thus we have $\pi_{1,\mu}=1$,\  $\pi_{2,\mu}=s_3$,\  $\pi_{3,\mu}=s_2s_3$\  and\  $\pi_{4,\mu}=s_1s_2s_3\,$. The map $\Psi_{\mu}$ is given by:
\[g_1\mapsto\left(\begin{array}{cccc}T_1 & \cdot & \cdot & \cdot\\ \cdot & T_1 & \cdot & \cdot\\ \cdot & \cdot & \cdot & u\\ \cdot & \cdot & u & \cdot\end{array}\right)\,,\ \ \ \ 
g_2\mapsto\left(\begin{array}{cccc}T_2 & \cdot & \cdot & \cdot\\ \cdot & \cdot & u & \cdot\\ \cdot & u & \cdot & \cdot\\ \cdot & \cdot & \cdot & T_1\end{array}\right)\,,\ \ \ \ 
g_3\mapsto\left(\begin{array}{cccc}\cdot & u & \cdot & \cdot\\ u & \cdot & \cdot & \cdot\\ \cdot & \cdot & T_2 & \cdot\\ \cdot & \cdot & \cdot & T_2\end{array}\right)\ ;\]
\[t_1\mapsto\text{Diag}(\xi_1,\xi_1,\xi_1,\xi_2)\,,\ \ \ t_2\mapsto\text{Diag}(\xi_1,\xi_1,\xi_2,\xi_1)\,,\ \ \  t_3\mapsto\text{Diag}(\xi_1,\xi_2,\xi_1,\xi_1)\,,\ \ \ t_4\mapsto\text{Diag}(\xi_2,\xi_1,\xi_1,\xi_1)\,;\]
\[e_1\mapsto\text{Diag}(1,1,0,0)\,,\ \ \ \ e_2\mapsto\text{Diag}(1,0,0,1)\,,\ \ \ \ e_3\mapsto\text{Diag}(0,0,1,1)\ .\]

\vskip .2cm
\noindent $\bullet$
 Let $\mu=(1,3)$. Then $m_{\mu}=4$ and $\cH^{\mu}\cong\cH_1\otimes\cH_3$ and we identify it below with $\cH_3$. We order the characters in the orbit $\mathcal{O} (\mu)$ as follows: 
$$\chi^{\mu}_1=(\xi_1,\xi_2,\xi_2,\xi_2)\,,\ \ \ \ \chi^{\mu}_2=(\xi_2,\xi_1,\xi_2,\xi_2)\,,\ \ \ \ \chi^{\mu}_3=(\xi_2,\xi_2,\xi_1,\xi_2)\ \ \ \text{and}\ \ \ \chi^{\mu}_4=(\xi_2,\xi_2,\xi_2,\xi_1)\ .$$
Thus we have $\pi_{1,\mu}=1$,\  $\pi_{2,\mu}=s_1$,\  $\pi_{3,\mu}=s_2s_1$\  and\  $\pi_{4,\mu}=s_3s_2s_1\,$. The map $\Psi_{\mu}$ is given by:
\[g_1\mapsto\left(\begin{array}{cccc}\cdot & u & \cdot & \cdot\\ u & \cdot & \cdot & \cdot\\ \cdot & \cdot & T_1 & \cdot\\ \cdot & \cdot & \cdot & T_1\end{array}\right)\,,\ \ \ \ 
g_2\mapsto\left(\begin{array}{cccc}T_1 & \cdot & \cdot & \cdot\\ \cdot & \cdot & u & \cdot\\ \cdot & u & \cdot & \cdot\\ \cdot & \cdot & \cdot & T_2\end{array}\right)\,,\ \ \ \ 
g_3\mapsto\left(\begin{array}{cccc}T_2 & \cdot & \cdot & \cdot\\ \cdot & T_2 & \cdot & \cdot\\ \cdot & \cdot & \cdot & u\\ \cdot & \cdot & u & \cdot\end{array}\right)\ ;\]
\[t_1\mapsto\text{Diag}(\xi_1,\xi_2,\xi_2,\xi_2)\,,\ \ \ t_2\mapsto\text{Diag}(\xi_2,\xi_1,\xi_2,\xi_2)\,,\ \ \  t_3\mapsto\text{Diag}(\xi_2,\xi_2,\xi_1,\xi_2)\,,\ \ \ t_4\mapsto\text{Diag}(\xi_2,\xi_2,\xi_2,\xi_1)\,;\]
\[e_1\mapsto\text{Diag}(0,0,1,1)\,,\ \ \ \ e_2\mapsto\text{Diag}(1,0,0,1)\,,\ \ \ \ e_3\mapsto\text{Diag}(1,1,0,0)\ .\]

\vskip .2cm
\noindent $\bullet$
 Let $\mu=(2,2)$. Then $m_{\mu}=6$ and $\cH^{\mu}\cong\cH_2\otimes\cH_2$. We order the characters in the orbit $\mathcal{O} (\mu)$ as follows: 
$$\begin{array}{c}\chi^{\mu}_1=(\xi_1,\xi_1,\xi_2,\xi_2)\,,\ \ \ \ \chi^{\mu}_2=(\xi_1,\xi_2,\xi_1,\xi_2)\,,\ \ \ \ \chi^{\mu}_3=(\xi_2,\xi_1,\xi_1,\xi_2)\,,\\[0.4em]
\chi^{\mu}_4=(\xi_1,\xi_2,\xi_2,\xi_1)\,,\ \ \ \ \chi^{\mu}_5=(\xi_2,\xi_1,\xi_2,\xi_1)\,,\ \ \ \ \chi^{\mu}_6=(\xi_2,\xi_2,\xi_1,\xi_1)\,.\end{array}$$
Thus we have $\pi_{1,\mu}=1$,\  $\pi_{2,\mu}=s_2$,\  $\pi_{3,\mu}=s_1s_2$,\  $\pi_{4,\mu}=s_3s_2$,\  $\pi_{5,\mu}=s_1s_3s_2$\  and\  $\pi_{6,\mu}=s_2s_1s_3s_2\,$. The map $\Psi_{\mu}$ is given by (where $T'_1:=T_1\otimes 1$ and $T''_1:=1\otimes T_1$):
\[g_1\mapsto\!\!\left(\begin{array}{cccccc}T'_1 & \cdot & \cdot & \cdot & \cdot & \cdot\\ \cdot & \cdot & u & \cdot & \cdot & \cdot\\ \cdot & u & \cdot & \cdot & \cdot & \cdot\\ \cdot & \cdot & \cdot & \cdot & u & \cdot\\ \cdot & \cdot & \cdot & u & \cdot & \cdot\\ \cdot & \cdot & \cdot & \cdot & \cdot & T''_1\end{array}\right)\!,\ \ \ \ 
g_2\mapsto\!\!\left(\begin{array}{cccccc}\cdot & u & \cdot & \cdot & \cdot & \cdot\\ u & \cdot & \cdot & \cdot & \cdot & \cdot\\ \cdot & \cdot & T'_1 & \cdot & \cdot & \cdot\\ \cdot & \cdot & \cdot & T''_1 & \cdot & \cdot\\ \cdot & \cdot & \cdot & \cdot & \cdot & u\\ \cdot & \cdot & \cdot & \cdot & u & \cdot\end{array}\right)\!,\ \ \ \ 
g_3\mapsto\!\!\left(\begin{array}{cccccc}T''_1 & \cdot & \cdot & \cdot & \cdot & \cdot\\ \cdot & \cdot & \cdot & u & \cdot & \cdot\\ \cdot & \cdot & \cdot & \cdot & u & \cdot\\ \cdot & u & \cdot & \cdot & \cdot & \cdot\\ \cdot & \cdot & u & \cdot & \cdot & \cdot\\ \cdot & \cdot & \cdot & \cdot & \cdot & T'_1\end{array}\right)\!;\]
\[\begin{array}{c}t_1\mapsto\text{Diag}(\xi_1,\xi_1,\xi_2,\xi_1,\xi_2,\xi_2)\,,\ \ \ t_2\mapsto\text{Diag}(\xi_1,\xi_2,\xi_1,\xi_2,\xi_1,\xi_2)\,,\\[0.4em]
t_3\mapsto\text{Diag}(\xi_2,\xi_1,\xi_1,\xi_2,\xi_2,\xi_1)\,,\ \ \ 
t_4\mapsto\text{Diag}(\xi_2,\xi_2,\xi_2,\xi_1,\xi_1,\xi_1)\,;\end{array}\]
\[e_1\mapsto\text{Diag}(1,0,0,0,0,1)\,,\ \ \ \ e_2\mapsto\text{Diag}(0,0,1,1,0,0)\,,\ \ \ \ e_3\mapsto\text{Diag}(1,0,0,0,0,1)\ .\]
\hfill$\triangle$
\end{Exa}

\subsection{Natural inclusions of subalgebras}\label{natural}

We recall that, for any $n\geq1$, the algebra $Y_{d,n}$ is naturally embedded into $Y_{d,n+1}$, as the subalgebra generated by $t_1,\dots,t_n,g_1,\dots,g_{n-1}$. If $x\in Y_{d,n}$, we will abuse notation and write also $x$ for the corresponding element of $Y_{d,n+n'}$, $n'\geq1$. Very often the context will make clear where $x$ lives, and otherwise we will specify it explicitly.

\vskip .1cm
Let $n\geq1$ and $\mu=(\mu_1,\dots,\mu_d)\in\Comp_d(n)$. For any $\mu'\geq\mu$ for the natural order on compositions (namely, $\mu'_a\geq\mu_a$ for $a=1,\dots,d$), we have a natural embedding of $\cH^{\mu}$ into $\cH^{\mu'}$. Explicitly, using the isomorphisms 
 $$\cH^{\mu}\simeq \cH_{\mu_1} \otimes \ldots \otimes \cH_{\mu_d}\ \ \ \textrm{ and }\ \ \ \cH^{\mu'}\simeq \cH_{\mu_1 '} \otimes \ldots \otimes \cH_{\mu_d '}\ ,$$
the embedding is given by
\[\ \cH_{\mu_1} \otimes \ldots \otimes \cH_{\mu_d}\ni x_1\otimes\dots\otimes x_d\mapsto x_1\otimes\dots\otimes x_d\in \cH_{\mu_1 '} \otimes \ldots \otimes \cH_{\mu_d '} .\]
When $\mu'=\mu^{[a]}$ for $a\in\{1,\ldots,d\}$, see (\ref{def-mu-a+}), the natural embedding $\cH^{\mu}\subset \cH^{\mu^{[a]}}$ is expressed on the basis $\{\tT_w\,,\ w\in\mS^{\mu}\}$ of $\cH^{\mu}$   by
\begin{equation}\label{inc-Hmu}
\cH^{\mu}\ni \tT_w \mapsto \tT_{(\mu_1+\dots+\mu_a+1,\dots,n-1,n)\,w\,(\mu_1+\dots+\mu_a+1,\dots,n-1,n)^{-1}}\in\cH^{\mu^{[a]}}\ ,
\end{equation}
where $(\mu_1+\dots+\mu_a+1,\dots,n-1,n)$ is the cyclic permutation on $\mu_1+\dots+\mu_a+1,\dots,n-1,n$.

\paragraph{Inclusion of basis elements.} Let $E_{\chi_k^{\mu}}\tg_w$ be an element of the basis of $Y_{d,n}$, where $\mu\in\Comp_d(n)$, $k\in\{1,\dots,m_{\mu}\}$ and $w\in\mS_n$.

For $a\in\{1,\dots,d\}$, denote by $k_a$ the integer in $\{1,\dots,m_{\mu^{[a]}}\}$ such that $\chi_{k_a}^{\mu^{[a]}}\in \operatorname{Irr}(\mathcal{A}_{n+1})$ is the character given by
\[\chi_{k_a}^{\mu^{[a]}}(t_i)=\chi_{k}^{\mu}(t_i)\,,\ \ \text{if $i=1,\dots,n$,}\quad\ \ \ \ \text{and}\ \ \ \ \quad\chi_{k_a}^{\mu^{[a]}}(t_{n+1})=\xi_a\ .\]
The characters $\{\chi_{k_a}^{\mu^{[a]}}\,,\ a=1,\dots,d\}$ are all the irreducible characters of $\mathcal{A}_{n+1}$ containing $\chi_k^{\mu}$ in their restriction to $\mathcal{A}_n$, and therefore we have $E_{\chi_k^{\mu}}=\sum_{1\leq a \leq d} E_{\chi_{k_a}^{\mu^{[a]}}}\ $ in $\mathcal{A}_{n+1}\,$. Thus, in $Y_{d,n+1}$, we have:
\begin{equation}\label{inc-basis}
E_{\chi_k^{\mu}}\tg_w=\sum_{1\leq a \leq d} E_{\chi_{k_a}^{\mu^{[a]}}}\tg_w\ .
\end{equation}

\paragraph{A formula for $\pi_{k_a,\mu^{[a]}}$.} Let $a\in\{1,\dots,d\}$. We recall that $\pi_{k,\mu}$ is defined as the element of $\mS_n$ of minimal length such that $\pi_{k,\mu}(\chi^{\mu}_1)=\chi_k^{\mu}$, and similarly, $\pi_{k_a,\mu^{[a]}}$ is the element of $\mS_{n+1}$ of minimal length such that $\pi_{k_a,\mu^{[a]}}(\chi^{\mu^{[a]}}_1)=\chi_{k_a}^{\mu^{[a]}}$. Writing symbolically a character $\chi\in\mathcal{A}_{n+1}$ as the collection $(\chi(t_1),\dots,\chi(t_{n+1}))$, we have
\[\chi^{\mu^{[a]}}_1=(\underbrace{\xi_1,\dots,\xi_1}_{\mu_1},\dots,\underbrace{\xi_a,\dots,\xi_a}_{\mu_a+1},\dots,\underbrace{\xi_d\dots,\xi_d}_{\mu_d})\quad\ \ \ \ \text{and}\quad\ \ \ \ \chi^{\mu^{[a]}}_{k_a}=(\chi^{\mu}_k,\xi_a)\ ,\]
so that  the last occurrence of $\xi_a$ in $\chi^{\mu^{[a]}}_1$ is in position $\mu_1+\dots+\mu_a+1$. Also, $\chi_1^{\mu}$ is obtained from $\chi^{\mu^{[a]}}_1$ by removing this last $\xi_a$. It is thus straightforward to see that
\begin{equation}\label{form-pi}
\pi_{k_a,\mu^{[a]}}=\pi_{k,\mu}\cdot(\mu_1+\dots+\mu_a+1,\dots,n,n+1)^{-1}\ .
\end{equation}
In the remaining of the paper, to simplify notations,  we will often write $\pi_k=\pi_{k,\mu}$ if there is no ambiguity on the choice of $\mu$ and also $\pi_{k_a}:=\pi_{k_a,\mu^{[a]}}$, for any $k\in\{1,\dots,m_{\mu}\}$ and $a\in\{1,\dots,d\}$.

\paragraph{Inclusion of matrix algebras.} The successive compositions of the isomorphism $\Phi_n$, the natural embedding of $Y_{d,n}$ in  $Y_{d,n+1}$ and the isomorphism $\Psi_{n+1}=\Phi_{n+1}^{-1}$ gives the embedding $\iota$ of the following diagram:
\[
\begin{array}{ccc}
 Y_{d,n+1} & \xrightarrow{\ \ \Psi_{n+1}\ \ } &\displaystyle\bigoplus_{\mu \in \Comp_d (n+1) }\Mat_{m_{\mu}} (\cH^{\mu})\\
 \bigcup & & \left\uparrow \rule{0pt}{0.5cm}\right.\iota\\
 Y_{d,n} & \xleftarrow{\ \ \Phi_{n}\ \ } &\displaystyle\bigoplus_{\mu \in \Comp_d (n) }\Mat_{m_{\mu}} (\cH^{\mu})
\end{array}
\]

In the formula below, an element $x\in\cH^{\mu}$ is also seen as an element of $\cH^{\mu^{[a]}}$, for any $a\in\{1,\dots,d\}$, via the natural embeddings recalled above. We keep the same notation $x$.
\begin{Prop}\label{prop-inc}
The embedding $\iota$ is given by $\iota=\bigoplus_{\mu \in \Comp_d (n)}\iota_{\mu}$, where the injective morphisms $\iota_{\mu}$ are given by:
\[\begin{array}{cccc}
\iota_{\mu}\ : & \Mat_{m_{\mu}} (\cH^{\mu}) & \to &\displaystyle\bigoplus_{1\leq a \leq d}\Mat_{m_{\mu^{[a]}}} (\cH^{\mu^{[a]}}) \\
 & x\,\M_{i,j} & \mapsto & \displaystyle\sum_{1 \leq a \leq d}\ x\,\M_{i_a,j_a}
 \end{array},\]
for any $\mu \in \Comp_d (n)$, any $x\in\cH^{\mu}$ and any $i,j\in\{1,\dots,m_{\mu}\}$.
\end{Prop}
\begin{proof}
 Let $E_{\chi_i^{\mu}}\tg_w$ be an element of the basis of $Y_{d,n}$, where $\mu\in\Comp_d(n)$, $i\in\{1,\dots,m_{\mu}\}$ and $w\in\mS_n$. Let $j\in\{1,\dots,m_{\mu}\}$ be uniquely determined by $w(\chi_i^{\mu})=\chi_j^{\mu}$. We have
 \[\Phi_n^{-1}(E_{\chi_i^{\mu}}\tg_w)=\tT_{\pi_i^{-1} w \pi^{}_j}\,\M_{i,j}\ .\]
 
On the other hand, we have, using (\ref{inc-basis}),
 $$\begin{array}{rcl}
 \Psi_{n+1}(E_{\chi_i^{\mu}}\,\tg_w)&=&\displaystyle \sum_{1\leq a \leq d}\Psi_{n+1}(E_{\chi_{i_a}^{\mu^{[a]}}}\,\tg_w)\\[1.4em]
 &=&\displaystyle\sum_{1\leq a \leq d}\tT_{\pi_{i_a}^{-1} w \pi_{j_a}}\,\M_{i_a,j_a}\ ,
 \end{array}$$
since, for any $a\in\{1,\dots,d\}$, the integer $j_a\in\{1,\dots,m_{\mu^{[a]}}\}$ is indeed such that $w(\chi_{j_a}^{\mu^{[a]}})=\chi_{i_a}^{\mu^{[a]}}$ (as $w\in\mS_n$).  
So we only have to check that $\tT_{\pi_{i_a}^{-1} w \pi^{}_{j_a}}$ is the image in $\cH^{\mu^{[a]}}$ of $\tT_{\pi_i^{-1} w \pi^{}_j}\in\cH^{\mu}$ under the natural embedding. This is an immediate consequence of (\ref{inc-Hmu}) and (\ref{form-pi}).
\end{proof}

\section{Applications to representation theory}\label{sec-rep}

This section presents the first applications of the isomorphism theorem described in the preceding section.
First we study the consequences on the representation theory of $Y_{d,n} $ and then we concentrate on the symmetric structure of this algebra. 
\subsection{Simple modules}\label{simple}
The case of an isomorphism between an algebra and a matrix algebra is a classical example of Morita equivalence which will be discussed in the next subsection. In fact, in this case,  the equivalence  is explicit. In particular,  due to the isomorphism, any simple $Y_{d,n}$-modules is of the form

$$(M_1\otimes \ldots \otimes M_d)^{m_{\mu}}\ ,$$ 
where $\mu=(\mu_1,\ldots,\mu_d)$ is a $d$-composition of $n$ and $M_a$ is a simple module of $\mathcal{H}_{\mu_a}$, for each $a=1,\ldots,d$ (see for example \cite[\S 17.B]{lam}).

\vskip .1cm
As the representation theory of the Iwahori--Hecke algebra is quite well understood (at least in characteristic $0$, see for example 
 \cite[ch. 8,9,10]{GP} for the semisimple case and \cite{GJ} for the modular case), we can deduce 
 from Theorem  \ref{main} the following results:
 \begin{itemize}
 \item Let $\theta : A \to k$ be a specialization to a field $k$ such that $\theta (u^2)=1$ and $\theta (v)=q-q^{-1}$ 
  for an element in $q\in k^{\times}$. Let $k_{\theta} Y_{d,n} $ be the specialized algebra 
  then $k_{\theta} Y_{d,n} $  is split semisimple if and only if  for all ${\mu \in \operatorname{Comp}_d (n) }$, the algebra  $k_{\theta} \mathcal{H}^{\mu} $
  is split semisimple. By \cite[Ex. 3.1.19]{GJ}, this happens if and only if :
 $$ \prod_{1\leq m \leq n} (1+ q^2+\ldots +  q^{2m-2})\neq 0$$
 we thus recover the semisimplicity criterion found in \cite[\S 6]{CPA1}. 
 \item The simple  $k_{\theta} Y_{d,n} (q)$-modules are  naturally labelled by the set of $d$-partitions of rank $n$ when the algebra is split semisimple. Moreover, in the non semisimple case, if we set 
 $$e:=\textrm{min} \{ i>0\ |\ 1+q^2+\ldots +q^{2i-2} =0\}$$
 then the simple modules are labelled by  the $d$-tuples of partitions such that each partition is $e$-regular. 
 \item The irreducible characters are completely determined by the irreducible characters of the Iwahori--Hecke algebra of type $A$. 
 For $M\in  \operatorname{Irr} ( k_{\theta}  \mathcal{H}^{\mu})$ with character $\chi_M$, 
   the  character of the simple $k_{\theta} Y_{d,n}$-module $(M)^{m_{\mu}}$ is given by :
   $$\chi (h)= \chi_M \circ \textrm{Tr}_{\operatorname{Mat}_{m_{\mu}}} \circ \Psi_{\mu} (h)\ ,\ \ \ \ \ \ \ \text{for $h\in k_{\theta}Y_{d,n}$,}$$
 where $\textrm{Tr}_{\operatorname{Mat}_{m_{\mu}}}$ is the usual  trace function on the matrix algebra.  
   In particular,  
 the decomposition matrices of the Yokonuma--Hecke algebra are entirely determined by the decomposition matrices of the Iwahori--Hecke algebra of type $A$. 
 \end{itemize}

\subsection{A Morita equivalence}

 From Theorem \ref{main}, we can thus deduce (see for example \cite[Ch. 17]{lam})  a Morita equivalence between the Yokonuma--Hecke algebra and a direct sum of Iwahori--Hecke algebras of type $A$ over any ring.
\begin{Prop}
Let $R$ be a commutative ring and $\theta :\C[u^{\pm1},v] \to R$ be a specialization. Then 
the algebra $R_{\theta}Y_{d,n}$  is Morita equivalent to $\bigoplus_{\mu \in \operatorname{Comp}_d (n) } R_{\theta}\mathcal{H}^{\mu} $.
\end{Prop}

 In addition, consider the Hecke algebra of the complex reflection group $G(d,1,n)$ (also known as Ariki--Koike algebra). Let $R$ be a commutative ring with unit and let 
 ${\bf Q}:=(Q_0,\ldots,Q_{d-1} )\in R^d$ and $x\in R^{\times}$.  The Hecke algebra of $G(d,1,n)$ is the  $R$-algebra $\mathcal{H}_n^{{\bf Q},x}$ with generators 
 $$T_0,T_1, \ldots, T_{n-1},$$
 and relations :
\begin{align*}
& T_{i}T_{i+1}T_{i}=T_{i+1}T_{i}T_{i+1}\ (i=1,\ldots ,n-2), \\
& T_{i}T_{j}=T_{j}T_{i}\ (|j-i|>1), \\
& (T_{i}-x)(T_{i}+x^{-1})=0\ (i=1,\ldots ,n-1).\\
&(T_0-Q_0)(T_0-Q_1)\ldots (T_0-Q_{d-1})=0
\end{align*}
\begin{Rem}
Note that  if $d=1$ and $R=\mathbb{C}[x,x^{-1}] $ then $\mathcal{H}_n^{{\bf Q},x}$ is nothing but the Iwahori--Hecke algebra of type $A_{n-1}$ with parameter $x$  of Remark \ref{RemA}. 
 \hfill$\triangle$
\end{Rem}
Now assume in addition that for all $a\neq b$ and $-n<i<n$ the element $x^{2i} Q_a-Q_b$ is an invertible element of $R$. 
By \cite{Dima,Durui}, over $R$, $\mathcal{H}_n^{{\bf Q},x}$ 
  is Morita equivalent to $\bigoplus_{\mu \in \operatorname{Comp}_d (n) }R_{\theta} \mathcal{H}^{\mu} $. 
 We thus deduce the following result. 
\begin{Cor}
Under the above hypothesis, assume that  $\theta :A \to R$  is a specialization such that $\theta (u^2)=1$ and $\theta (v)=x-x^{-1}$  then
$R_{\theta}Y_{d,n} $ 
is Morita equivalent to $\mathcal{H}_n^{{\bf Q},x} $.
\end{Cor}

\subsection{Symmetrizing form and Schur elements}

We now study the symmetric structure of the Yokonuma--Hecke algebra. The algebra $Y_{d,n}$ is symmetric and thus it has a symmetrizing form which controls part of its representation theory.  We will in particular recover results obtained in \cite{CPA1} concerning the symmetric structure of  $Y_{d,n}$. In fact, the isomorphism theorem will also give an explanation and a  new  interpretation of these results.  

\paragraph{Preliminaries on symmetric algebras.} We recall that a symmetric algebra $H$ over a commutative ring $R$ is an $R$-algebra  equipped with a trace function 
$$\utau: H\to R$$
such that the bilinear form 
$$\begin{array}{rcl}
H\times H & \to & R \\
(h_1,h_2) & \mapsto & \utau (h_1 h_2)
\end{array}$$
is non-degenerate. We refer to \cite[Ch. 7]{GP} for a study of the properties of the symmetric algebras. In particular, if $K$ is a field containing $R$ and such that $KH$ is split semisimple, then for all 
 ${V\in \textrm{Irr} (KH)}$, there exists  $s_{V}$ in the integral closure of $R$ in $K$ such that 
 $$\utau=\sum_{\chi\in \textrm{Irr} (H)} (1/s_{V})\,\chi_V\ ,$$
where $\utau$ is extended to $KH$ and $\chi_V$ is the character associated to $V$.  The elements $s_{V}$ are called the Schur elements associated to $\utau$ and they are known to control part of the representation theory of $H$. We will use the following general result.
\begin{lemma}\label{lem-sym}
\begin{itemize}
\item[{\rm (i)}]  Let $N\in \mathbb{Z}_{>0}$. The algebra 
 $\operatorname{Mat}_N (H)$ is a symmetric algebra with symmetrizing form
$\utau^{\operatorname{mat}}:=\utau \circ \operatorname{Tr}_{\operatorname{Mat}_N}$\,,
where $\operatorname{Tr}_{\operatorname{Mat}_N}$ is the usual trace function on $\operatorname{Mat}_N (H)$.

\item[{\rm (ii)}] Let $M$ be a simple $KH$-module and $s_M$ its Schur element associated to $\utau$. Then the Schur element $s_M^{\operatorname{mat}}$ of the simple $\Mat_N(KH)$-module $M^N$ associated to $\utau^{\operatorname{mat}}$ is equal to $s_M$.
\end{itemize}
\end{lemma}
\begin{proof}
(i) The form $\utau^{\operatorname{mat}}$ is clearly a trace function. All we have to do is to check that this is non-degenerate. 
 Let $b_1\in   \operatorname{Mat}_N (H)$  and assume that for all $b_2\in \operatorname{Mat}_N (H)$, we have
  $\utau^{\operatorname{mat}} (b_1. b_2)=0$. Let $h\in H$ and  consider the element $b_3:=h.\textrm{Id}_N \in  \operatorname{Mat}_N (H)$ where $\textrm{Id}_N$ is the identity matrix in $\operatorname{Mat}_N (H)$. 
Then we have 
   $$\begin{array}{rcl}
    \utau^{\operatorname{mat}} (b_1. b_2  b_3)&=&\utau\circ  \operatorname{Tr}_{\operatorname{Mat}_N} (hb_1.b_2) \\
    &=& \utau (h.  \operatorname{Tr}_{\operatorname{Mat}_N} (b_1.b_2))
    \end{array}
$$
   As this element is zero for all $h\in H$ and as $\utau$ is a symmetrizing trace, we have $ \operatorname{Tr}_{\operatorname{Mat}_N} (b_1.b_2)=0$
 for all $b_2\in \operatorname{Mat}_N (H)$. This implies that $b_1=0$. 
 
 \vskip .1cm
 (ii) Let $E_M$ be a primitive idempotent of $KH$ associated to the simple module $M$. Then, by definition, we have $\utau(E_M)=1/s_M$. Now let $E_M^{\operatorname{mat}}\in\Mat_N(KH)$ be the matrix with $E_M$ in position $(1,1)$ and $0$ everywhere else. Then $E_M^{\operatorname{mat}}$ is a primitive idempotent of $\Mat_N(KH)$ associated to the simple module $M^N$. Thus, we calculate
 \[1/s_M^{\operatorname{mat}}=\utau^{\operatorname{mat}}(E_M^{\operatorname{mat}})=\utau\circ\operatorname{Tr}_{\operatorname{Mat}_N}(E_M^{\operatorname{mat}})=\utau(E_M)=1/s_M\ ,\]
which is the desired result.
\end{proof}

\paragraph{Symmetric structure of $Y_{d,n}$\,.}
The Iwahori--Hecke algebra of type $A$ is a symmetric algebra with symmetrizing form 
 $\utau_n: \mathcal{H}_n \to \C[u^{\pm1},v] $ given on the basis elements by 
$$\utau_n (\tT_w )=\left\{ \begin{array}{ll}
 1 & \text{if $w=1$,}\\
 0 & \text{otherwise.}
\end{array}\right.$$
This, in turn, implies the existence of a symmetrizing form for all $\mu=(\mu_1,\ldots, \mu_d)\in \operatorname{Comp}_d(n)$ on $\mathcal{H}^{\mu} $ by restriction:
 $\utau^{\mu}: \mathcal{H}^{\mu} \to \C[u^{\pm1},v]$. 
Seeing as usual $\mathcal{H}^{\mu}$ as $\cH_{\mu_1}\otimes\ldots\otimes\cH_{\mu_d}$, this linear form satisfies, for all $(w_1,\ldots,w_d)\in \mathfrak{S}_{\mu_1}  \times \ldots \times \mathfrak{S}_{\mu_d}$,  
\begin{equation}\label{tau-mu}
\utau^{\mu} (\tT_{w_1}\otimes\ldots\otimes\tT_{w_d})=\utau_{\mu_1} (\tT_{w_1})\ldots \utau_{\mu_d} (\tT_{w_d})\ .
\end{equation}
For any $n\geq1$ and $\lambda$ a partition of $n$, let $M_{\lambda}$ be the simple module of the split semisimple algebra $\C(u,v)\cH_n$. We denote by $s_{\lambda}:=s_{M_{\lambda}}$ the Schur element of $M_{\lambda}$ associated to $\utau$. We also set $s_{\emptyset}:=1$.

Now let $\ulambda=(\lambda^1,\ldots,\lambda^{d})$ be a $d$-tuple of partitions such that $\mu=(|\lambda^1|,\ldots,|\lambda^{d}|)$. Then $M_{\lambda_1}\otimes \ldots\otimes M_{\lambda_d}$ is a simple $\C(u,v)\cH^{\mu}$-module and, from (\ref{tau-mu}), its Schur element associated to $\utau^{\mu}$, denoted by $s_{\ulambda}$, is given by:
\begin{equation}\label{schur}
s_{\ulambda}=s_{\lambda^1}\ldots s_{\lambda^d}\ .
\end{equation}

Finally, from Lemma \ref{lem-sym}(i), we obtain a symmetrizing form on the algebra 
 $\bigoplus_{\mu \in \operatorname{Comp}_d (n)} \operatorname{Mat}_{m_{\mu}} (\mathcal{H}^{\mu})$ given by $\bigoplus_{\mu \in \operatorname{Comp}_d (n)} \utau^{\mu} \circ \operatorname{Tr}_{\operatorname{Mat}_{m_{\mu}}}$\ . Moreover, from Lemma \ref{lem-sym}(ii), the associated Schur element of the simple module indexed by a $d$-partition $\ulambda=(\lambda^1,\ldots,\lambda^{d})$ of $n$  is given by Formula (\ref{schur}).

\vskip .1cm
Let us come back to the Yokonuma--Hecke algebra $Y_{d,n}$. By the discussion in Subsection \ref{simple}, we know that the simple $\mathbb{C}(u,v)Y_{d,n}$-modules are labelled by the set of $d$-partitions of $n$. From the preceding discussion together with the isomorphism theorem (Theorem \ref{main}), we obtain naturally a symmetrizing  form on  $Y_{d,n}$ given explicitly by 
$$\urho_n:=\bigoplus_{\mu \in \operatorname{Comp}_d (n)} \utau^{\mu} \circ \operatorname{Tr}_{\operatorname{Mat}_{m_{\mu}}}\circ\Psi_{\mu}\ ,$$
and moreover, the associated Schur elements are given by Formula (\ref{schur}). We note that the symmetrizing form $\urho_n$ is also considered in \cite[\S 34]{lu}.

\paragraph{Alternative formula for $\urho_n$\,.}
In \cite{CPA1}, it is proved that the following formula defines a symmetrizing form $\widetilde{\urho}_n : Y_{d,n} \to \C[u^{\pm1},v] $ on $Y_{d,n}$:
$$\widetilde{\urho}_n (t_1^{a_1} \ldots t_n^{a_n} \tg_w )=\left\{ \begin{array}{ll}
 d^n & \text{ if $a_1=\ldots=a_n=1$ and $w=1$,}\\
 0 & \text{otherwise.}
\end{array}\right.$$
It turns out that the form $\widetilde{\urho}_n$ actually coincides with the natural symmetrizing form $\urho_n$.
 \begin{Prop}\label{prop-schur}
The form $\widetilde{\urho}_n$ coincides with the symmetrizing form $\urho_n$ on $Y_{d,n}$. 
 \end{Prop}
 \begin{proof}
 We study the values taken by the two traces on the basis given by 
  $$\{ E_{\chi} \tg_w \ |\ \chi \in  \operatorname{Irr} (\mathcal{A}_n),\ w\in \mathfrak{S}_n\}.$$
 So let us fix $\mu\in \operatorname{Comp}_d (n)$, $k\in \{1,\ldots, m_{\mu}\}$ and $w\in \mathfrak{S}_n$. We have, using Formula (\ref{E-chi}) for $E_{\chi_k^{\mu}}$,
 $$\begin{array}{rcl}
\widetilde{\urho}_n (E_{\chi_k^{\mu}} \tg_w)&=&\widetilde{\urho}_n \Biggl(\displaystyle
  \Bigl(\prod_{1\leq i\leq n} \frac{1}{d} \sum_{0\leq s \leq d-1} \chi^{\mu}_{k} (t_i)^s\,t_i^{-s} \Bigr)\ \tg_w\Biggr)\\[1.5em]
   &=&\widetilde{\urho}_n \Bigl(\displaystyle \bigl(\prod_{1\leq i\leq n} \frac{1}{d}\,\bigr)\ \tg_w  \Bigr) \\[1.5em]
 &=&\left\{ 
 \begin{array}{ll}
 1 & \text{ if $w=1$\,,}\\
 0 & \text{ otherwise.}
 \end{array}\right. 
 \\ 
 \end{array}$$
 On the other hand we have 
  $$\begin{array}{rcl}
\urho_n (E_{\chi_k^{\mu}} \tg_w)&=&
\utau^{{\mu}  } \circ \operatorname{Tr}_{\operatorname{Mat}_{m_{\mu}}}\circ{\Psi}_{\mu} (E_{\chi_k^{\mu}}\tg_w)\\[0.5em]
&=&
\utau^{{\mu}  } \circ \operatorname{Tr}_{\operatorname{Mat}_{m_{\mu}}} (\tT_{\pi_{k,\mu}^{-1} w\pi_{j,\mu}}\,\M_{k,j})
 \end{array}$$
where $j\in \{1,\ldots,m_{\mu}\}$ satisfies $w (\chi_j^{\mu})=\chi_k^{\mu}$. We have $j=k$ if and only if $\pi_{k,\mu}^{-1} w\pi_{k,\mu}\in\mS^{\mu}$. 
  We obtain :
  $$\urho_n (E_{\chi_k^{\mu}} \tg_w) =\left\{ 
 \begin{array}{rr}
 1 & \text{ if $\pi_{k,\mu}^{-1} w\pi_{k,\mu}=1 \iff w=1$,}\\
 0 & \text{ otherwise.}
 \end{array}\right. 
$$
 and this concludes the proof. 
 \end{proof} 

\begin{Rem}
The Schur elements associated to $\widetilde{\urho}_n$ were obtained in \cite{CPA1} by a direct calculation. From Proposition \ref{prop-schur} and the discussion before it, we recover the result, namely that the Schur elements associated to $\widetilde{\urho}_n$ are given by Formula (\ref{schur}). Furthermore, we note that Proposition \ref{prop-schur} implies immediately the centrality and the non-degeneracy of $\widetilde{\urho}_n$ (which was also proved by direct calculations in \cite{CPA1}).
\end{Rem}

\section{Classification of Markov traces on Yokonuma--Hecke algebras}\label{sec-mark}

In this section, we use the isomorphism theorem to obtain a complete classification of Markov traces on $Y_{d,n}$. We use a definition of Markov traces analogous to the one in \cite[section 4.5]{GP} for the Iwahori--Hecke algebras of type $A$. From now on, we extend the ground ring to $\mathbb{C}[u^{\pm1},v^{\pm1}]$ and keep the same notations ($\cH_n,\cH^{\mu},Y_{d,n},...$) for the extended algebras.

\subsection{Markov traces on Iwahori--Hecke algebras of type A}

A Markov trace on the family of algebras $\{\mathcal{H}_n\}_{n\geq1}$ is a family of linear functions $\tau_n\ :\ \mathcal{H}_n\to \C[u^{\pm1},v^{\pm1}]$ ($n\geq1$) satisfying:
\begin{equation}\label{Markov-tau}
\begin{array}{lllr}
(\text{M}1) & \tau_n(xy)=\tau_n(yx)\ ,& \text{for $n\geq1$ and $x,y\in\mathcal{H}_n$;} & \text{\emph{(Trace condition)}}\\[0.2em]
(\text{M}2) & \tau_{n+1}(xT_n)=\tau_{n+1}(xT_n^{-1})=\tau_n(x)\ , & \text{for $n\geq1$ and $x\in\mathcal{H}_n$.} &\text{\emph{(Markov condition)}}
\end{array}
\end{equation}
It is a normalized Markov trace if it satisfies in addition
\begin{equation}\label{Markov-tau2}
\begin{array}{lllr}(\text{M}0) &\tau_1(1)=1\ . & & \text{\emph{(Normalization condition)}}
\end{array}
\end{equation}
In (M2) and in the following as well, we keep the same notation $x$ for an element of $\cH_n$ and the corresponding element of $\cH_{n+n'}$, $n'\geq1$, using the natural embedding of $\cH_n$ in $\cH_{n+n'}$.

\vskip .1cm
It is a classical result that a normalized Markov trace on $\{\mathcal{H}_n\}_{n\geq1}$ exists and is unique \cite[section 4.5]{GP}. From now on, $\{\tau_n\}_{n\geq1}$ will be this unique normalized Markov trace. For later use, we also set $\tau_0\ :\ \mathcal{H}_0:=\C[u^{\pm1},v^{\pm1}] \to \C[u^{\pm1},v^{\pm1}]$ to be the identity map on $\C[u^{\pm1},v^{\pm1}]$.

\vskip .1cm
As $T_n-u^2T_n^{-1}=v$ for any $n\geq1$, we have, using the Markov condition, that
\begin{equation}\label{rest-tau}
\tau_{n+1}(x)=v^{-1}(1-u^2)\tau_n(x)\ \ \ \ \ \text{for any $n\geq1$ and any $x\in\mathcal{H}_n$,}
\end{equation}
and by induction on $n$, using that $\tau_1(1)=1$, we obtain
\begin{equation}\label{tau1}
\tau_{n}(1)=\bigl(v^{-1}(1-u^2)\bigr)^{n-1}\ \ \ \ \ \text{for any $n\geq1$.}
\end{equation}

We will need later the following properties of the Markov trace $\{\tau_n\}_{n\geq1}$. For the second item, we recall that  $\cH^{\mu}\simeq \cH_{\mu_1}\otimes\dots\otimes\cH_{\mu_d}$  and these two algebras are identified. Recall also that this algebra is naturally embedded in $\cH_n$ for any $\mu=(\mu_1,\dots,\mu_d)\in\Comp_d(n)$. See (\ref{def-base-mu}) for the definition of $[\mu]$.
\begin{lemma}\label{lemme-tau-H}
\begin{enumerate}
\item[{\rm (i)}]  For any $n\geq 1$, we have:
\begin{equation}\label{tau-H-1}
\tau_{n+1}(xT_ky)=\tau_{n+1}(xT_k^{-1}y)=\tau_n(xy)\,,\ \ \ \ \ \text{for any $k\in\{1,\dots,n\}$ and $x,y\in\mathcal{H}_k$.}
\end{equation}
\item[{\rm (ii)}]   For any $n\geq1$ and any $\mu\in\Comp_d(n)$, we have:
\begin{equation}\label{tau-H-2}
\tau_{n}(x_1\otimes\dots\otimes x_d)=\bigl(v^{-1}(1-u^2)\bigr)^{\vert[\mu]\vert-1}\tau_{\mu_1}(x_1)\dots\tau_{\mu_d}(x_d)\,,\ \ \ \ \ \text{for any $x_1\otimes\dots\otimes x_d\in\mathcal{H}^{\mu}$.}
\end{equation}
\end{enumerate}
\end{lemma}
\begin{proof} (i)  Let $n\geq 1$, $k\in\{1,\dots,n\}$ and $x,y\in\mathcal{H}_k$. We proceed by induction on $n-k$. For $k=n$, Equation (\ref{tau-H-1}) follows directly from Conditions (M1) and (M2). Assume $k<n$. Then $T_{k+1}$ exists in $\cH_{n+1}$ and commutes with $x$ and $y$. By centrality of $\tau_{n+1}$, we have
\[\tau_{n+1}(xT_k^{\pm1}y)=\tau_{n+1}(xT^{}_{k+1}T_k^{\pm1}T_{k+1}^{-1}y)\ .\]
Using the braid relation  $T^{}_{k+1}T_k^{\pm1}T_{k+1}^{-1}=T_{k}^{-1}T_{k+1}^{\pm1}T^{}_{k}$ and the induction hypothesis, we conclude that
\[\tau_{n+1}(xT_k^{\pm1}y)=\tau_{n+1}(xT_{k}^{-1}T_{k+1}^{\pm1}T^{}_{k}y)=\tau_{n}(xT_{k}^{-1}T^{}_{k}y)=\tau_n(xy)\ .\]

(ii) Let $n\geq1$, $\mu\in\Comp_d(n)$ and $x=x_1\otimes\dots\otimes x_d\in\mathcal{H}^{\mu}$.  
First assume that $x=1$. We have, using (\ref{tau1}) and the convention $\tau_0(1)=1$,
\[\tau_{\mu_1}(1)\dots\tau_{\mu_d}(1)=\prod_{\mu_a\geq1}\bigl(v^{-1}(1-u^2)\bigr)^{\mu_a-1}=\bigl(v^{-1}(1-u^2)\bigr)^{n-\vert[\mu]\vert}\ ,\]
which yields, together with (\ref{tau1}), Formula (\ref{tau-H-2}) for $x=1$.

Let now $x\neq 1$. We proceed by induction on $n$ (the case $n=1$ being covered by the case $x=1$). Using that $x\neq 1$, we take $a\in\{1,\dots,d\}$ to be such that $x_{a+1}=\dots=x_d=1$ and $x_a=h_1T_kh_2\in\cH_{\mu_a}$ with $k\in\{1,\dots,\mu_a-1\}$ and $h_1,h_2\in\cH_{k}$ (in particular, $\mu_a\geq2$). We set $\nu=\mu_{[a]}\in\Comp_d(n-1)$ (that is, $\nu_a=\mu_a-1$ and $\nu_b=\mu_b$ if $b\neq a$). We calculate, using item (i),
\[\tau_n(x)=\tau_n(x_1\otimes\dots \otimes x_{a-1}\otimes h_1T_{k}h_2\otimes 1\dots\otimes 1) =\tau_{n-1}(x_1\otimes\dots \otimes x_{a-1} \otimes h_1h_2\otimes 1\dots\otimes 1)\ ;\]
using induction hypothesis, we then obtain
\[\tau_n(x)= \bigl(v^{-1}(1-u^2)\bigr)^{\vert[\nu]\vert-1}\tau_{\nu_1}(x_1)\dots\tau_{\nu_{a-1}}(x_{a-1})\tau_{\nu_a}(h_1h_2)\tau_{\nu_{a+1}}(1)\dots\tau_{\nu_d}(1)\ ;\]
finally, we have $[\nu]=[\mu]$, since $\mu_a\geq2$, and moreover, using item (i), $\tau_{\nu_a}(h_1h_2)=\tau_{\mu_a-1}(h_1h_2)=\tau_{\mu_a}(h_1T_kh_2)$. So we conclude that Formula (\ref{tau-H-2}) is satisfied.
\end{proof}

\subsection{Markov traces on Yokonuma--Hecke algebras}\label{subsec-Mark}

A Markov trace on the family of algebras $\{Y_{d,n}\}_{n\geq1}$ is a family of linear functions $\rho_n\ :\ Y_{d,n}\to \mathbb{C}[u^{\pm1},v^{\pm1}]$ ($n\geq1$) satisfying:
\begin{equation}\label{Markov-rho}
\begin{array}{lllr}
(\text{M}1) &\rho_n(xy)=\rho_n(yx)\ ,& \text{for any $n\geq1$ and $x,y\in Y_{d,n}$;} & \text{\emph{(Trace condition)}}\\[0.2em]
(\text{M}2) &\rho_{n+1}(xg_n)=\rho_{n+1}(xg_n^{-1})=\rho_n(x)\ , & \text{for any $n\geq1$ and $x\in Y_{d,n}$.} &\text{\emph{(Markov condition)}}
\end{array}
\end{equation}

\begin{Rem}\label{rem-norm}
For the Markov traces on the Iwahori--Hecke algebras, there is no loss of generality in considering the normalized Markov trace. Indeed, it is straightforward to see that if $\{\tau_n\}_{n\geq1}$ is a Markov trace on $\{\cH_n\}_{n\geq1}$ such that $\tau_1(1)=0$, then all the linear functions $\tau_n$ are identically 0. So we can assume that $\tau_1(1)\neq 0$ and normalize it so that $\tau_1(1)=1$. 

This remark is no longer valid for the Yokonuma--Hecke algebras, for which a Markov trace $\{\rho_n\}_{n\geq1}$ may satisfy $\rho_1(1)=0$ without being trivial (see the classification below). Therefore, we will work in the general setting of non-normalized Markov traces.\hfill$\triangle$
\end{Rem}

\vskip .1cm
Let $n\geq 1$ and let $\kappa$ be any linear function from $Y_{d,n}$ to $\mathbb{C}[u^{\pm1},v^{\pm1}]$ satisfying the trace condition $\kappa(xy)=\kappa(yx)$, $\forall x,y\in Y_{d,n}$. Recall the isomorphisms $\Phi_{\mu}$ and $\Psi_{\mu}=\Phi_{\mu}^{-1}$ between $\Mat_{m_{\mu}} (\cH_{\mu})$ and $E_{\mu}Y_{d,n}$ given by (\ref{def-phi})-(\ref{def-psi}). For each $\mu\in \Comp_d(n)$, the composed map $\kappa\circ\Phi_{\mu}$ is a linear map from $\Mat_{m_{\mu}} (\cH_{\mu})$ to $\mathbb{C}[u^{\pm1},v^{\pm1}]$ also satisfying the trace condition. As the usual trace of a matrix is the only trace function on a matrix algebra (up to normalization), the map $\kappa\circ\Phi_{\mu}$ is of the form:
\[\kappa\circ\Phi_{\mu}= \kappa^{\mu}\circ\operatorname{Tr}_{\Mat_{m_{\mu}}}\ ,\]
for some trace function $\kappa^{\mu}:\,\cH^{\mu}\to \mathbb{C}[u^{\pm1},v^{\pm1}]$.  In other words, we have
\[\kappa(x)=\sum_{\mu\in\Comp_d(n)}\kappa^{\mu}\circ\operatorname{Tr}_{\Mat_{m_{\mu}}}\circ\Psi_{\mu}(E_{\mu}x)\ ,\ \ \ \ \ \text{for $x\in Y_{d,n}$\,,}\]
(where we wrote $x=\sum_{\mu\in\Comp_d(n)}E_{\mu}x$) and we refer to the maps $\kappa^{\mu}$ as the trace functions associated to $\kappa$.

\paragraph{Classification of Markov traces on $\{Y_{d,n}\}_{n\geq1}$.}
We are now ready to give the classification of Markov traces on the Yokonuma--Hecke algebras $\{Y_{d,n}\}_{n\geq1}$ which is the main result of this section. We refer to Section \ref{sec-def}, (\ref{def-base-mu}) and (\ref{comp0}), for the definition of the base $[\mu]$ of a composition $\mu$ and of the set $\Comp_d^0$.

\begin{Th}\label{theo-mark}
A family $\{\rho_n\}_{n\geq1}$ of linear functions, $\rho_n\ :\ Y_{d,n}\to \mathbb{C}[u^{\pm1},v^{\pm1}]$, is a Markov trace on the family of algebras $\{Y_{d,n}\}_{n\geq1}$   if and only if there is a set of parameters $\{\alpha_{\mu^0}\,,\ \mu^0\in\Comp^0_d\}\subset \mathbb{C}[u^{\pm1},v^{\pm1}]$ such that 
\begin{equation}\label{rho-n}
\rho_n(x)=\sum_{\mu\in\Comp_d(n)}\rho^{\mu}\circ\operatorname{Tr}_{\Mat_{m_{\mu}}}\circ\Psi_{\mu}(E_{\mu}x)\,,\ \ \ \ \ \text{for $n\geq1$ and $x\in Y_{d,n}$,}
\end{equation}
where the associated trace functions $\rho^{\mu}\ :\ \cH^{\mu}\to \mathbb{C}[u^{\pm1},v^{\pm1}]$ are given by
\begin{equation}\label{rho-mu}
\rho^{\mu}=\alpha_{[\mu]}\cdot\tau_{\mu_1}\otimes\dots\otimes \tau_{\mu_d}\,,\ \ \ \ \ \text{for any $\mu\in\Comp_d(n)$.}
\end{equation}
\end{Th}

The remaining of this section is devoted to the proof of the Theorem.

\paragraph{Preliminary lemmas.}
Let $\{\rho_n\}_{n\geq1}$ be a family of linear functions $\rho_n\ :\ Y_{d,n}\to \mathbb{C}[u^{\pm1},v^{\pm1}]$ satisfying the trace condition (M1).

\begin{lemma}\label{lemme1}
The family $\{\rho_n\}_{n\geq1}$ satisfies the Markov condition (M2) if and only if the associated traces $\rho^{\mu}$ satisfy, for any $\mu\in\bigcup_{n\geq1}\Comp_d(n)$ and any $a\in\{1,\dots,d\}$ such that $\mu_a\geq1$,
\begin{equation}\label{mark-asso}
\begin{array}{rcl}
\rho^{\mu^{[a]}}(x_1\otimes\dots\otimes x_a T_{\mu_a}\otimes\dots\otimes x_d)&=&\rho^{\mu^{[a]}}(x_1\otimes\dots\otimes x_a T_{\mu_a}^{-1}\otimes\dots\otimes x_d)\\[0.4em]
&=&\rho^{\mu}(x_1\otimes\dots\otimes x_a \otimes\dots\otimes x_d),
\end{array}
\end{equation}
for any $x_1\otimes\dots\otimes x_d\in\cH^{\mu}$.
\end{lemma}
\begin{proof}
Let $n\geq1$ and $x\in Y_{d,n}$. In the proof, we will often use the notations and the results explained in subsection \ref{natural}. First, note that  it is enough to take $x=E_{\chi_i^{\mu}}\tg_w$, an element of the basis of $Y_{d,n}$, where $\mu\in\Comp_d(n)$, $i\in\{1,\dots,m_{\mu}\}$ and $w\in\mS_n$. For later use, we denote by $b$ the integer in $\{1,\dots,d\}$ such that $\chi_i^{\mu}(t_n)=\xi_b$.

We recall that $\Psi_{\mu}(x)=\tT_{\pi_i^{-1} w \pi^{}_j}\,\M_{i,j}$,  where $j\in\{1,\dots,m_{\mu}\}$ is uniquely determined by $w(\chi_j^{\mu})=\chi_i^{\mu}$. Thus, we have:
\begin{equation}\label{eq-rho1}
\begin{array}{rcl}
\rho_n(x)&=&\rho^{\mu}\circ\operatorname{Tr}_{\Mat_{m_{\mu}}}\circ\Psi_{\mu}(x)\\[0.4em]
&=&\left\{\begin{array}{ll}
0 & \text{if $w(\chi_i^{\mu})\neq\chi_i^{\mu}$\,;}\\[0.3em]
\rho^{\mu}(\tT_{\pi_i^{-1} w \pi^{}_i})\ \  & \text{if $w(\chi_i^{\mu})=\chi_i^{\mu}$\,.}
\end{array}\right.
\end{array}
\end{equation}
Now, in $Y_{d,n+1}$, we have (due to defining formulas (\ref{def-g-tg}) for $g_w$ and $\tg_w$)
$$xg_n=u\sum_{1\leq a \leq d} E_{\chi_{i_a}^{\mu^{[a]}}}\tg_{ws_n}\ ,$$
and we note that, for $1\leq a \leq d$, we have $ws_n(\chi_{i_a}^{\mu^{[a]}})=\chi_{i_a}^{\mu^{[a]}}$ if and only if $s_n(\chi_{i_a}^{\mu^{[a]}})=\chi_{i_a}^{\mu^{[a]}}$ and $w(\chi_{i_a}^{\mu^{[a]}})=\chi_{i_a}^{\mu^{[a]}}$ (because $w\in\mS_n$). It means that $ws_n(\chi_{i_a}^{\mu^{[a]}})=\chi_{i_a}^{\mu^{[a]}}$ if and only if $a=b$ and $w(\chi_i^{\mu})=\chi_i^{\mu}$. Thus we obtain:
$$\begin{array}{rcl}
\rho_{n+1}(xg_n)&=&u\displaystyle \sum_{1\leq a \leq d}\rho^{\mu^{[a]}}\circ\operatorname{Tr}_{\Mat_{m_{\mu^{[a]}}}}\circ\Psi_{\mu^{[a]}}(xg_n)\\[1.5em]
&=&\left\{\begin{array}{ll}
0 & \text{if $w(\chi_i^{\mu})\neq\chi_i^{\mu}$\,;}\\[0.3em]
u\,\rho^{\mu^{[b]}}(\tT_{\pi_{i_b}^{-1} ws_n \pi^{}_{i_b}})\ \  & \text{if $w(\chi_i^{\mu})=\chi_i^{\mu}$\,.}
\end{array}\right.  \end{array}$$
We write $\pi_{i_b}^{-1} ws_n \pi^{}_{i_b}=\pi_{i_b}^{-1} w\pi^{}_{i_b}\cdot\pi^{-1}_{i_b}s_n \pi^{}_{i_b}$. Recall that $\pi_{i_b}=\pi_i\cdot(\mu_1+\dots+\mu_b+1,\dots,n,n+1)^{-1}$, see (\ref{form-pi}); moreover, 
$$\pi_i(n+1)=n+1\,,\ \ \ \ \ \text{and}\ \ \ \ \ \pi_i(\mu_1+\dots+\mu_b)=n\ ,$$ 
since $\pi_i\in\mS_n$ and $\pi_i$ is the element of minimal length sending $\chi_1^{\mu}$ on $\chi_i^{\mu}$.
Therefore we have
$$\pi^{-1}_{i_b}s_n \pi^{}_{i_b}=(\mu_1+\dots+\mu_b,\ \mu_1+\dots+\mu_b+1)\ \ \ \ \ \text{and}\ \ \ \ \ \pi_{i_b}^{-1} w \pi^{}_{i_b}(\mu_1+\dots+\mu_b+1)=\mu_1+\dots+\mu_b+1\,,$$
We conclude that $\tT_{\pi_{i_b}^{-1} ws_n \pi^{}_{i_b}}=\tT_{\pi_{i_b}^{-1} w \pi^{}_{i_b}}\tT_{\mu_1+\dots+\mu_b}=u^{-1}\tT_{\pi_{i_b}^{-1} w \pi^{}_{i_b}}T_{\mu_1+\dots+\mu_b}$\,, and in turn
\begin{equation}\label{eq-rho2}
\rho_{n+1}(xg_n)=\left\{\begin{array}{ll}
0 & \text{if $w(\chi_i^{\mu})\neq\chi_i^{\mu}$\,;}\\[0.3em]
\rho^{\mu^{[b]}}(\tT_{\pi_{i_b}^{-1} w \pi^{}_{i_b}}T_{\mu_1+\dots+\mu_b})\ \  & \text{if $w(\chi_i^{\mu})=\chi_i^{\mu}$\,.}
\end{array}\right.
\end{equation}

Now we will calculate $\rho_{n+1}(xg_n^{-1})$ using $g_n^{-1}=u^{-2}(g_n-ve_n)$. First we note that $e_nx=\displaystyle\sum_{1\leq a \leq d}e_nE_{\chi_{i_a}^{\mu^{[a]}}}\tg_{w}$, which gives $e_nx=E_{\chi_{i_b}^{\mu^{[b]}}}\tg_{w}$\,, since $e_nE_{\chi}=0$ whenever $\chi(t_n)\neq\chi(t_{n+1})$.
In addition we have $w(\chi_{i_b}^{\mu^{[b]}})=\chi_{i_b}^{\mu^{[b]}}$ if and only if $w(\chi_i^{\mu})=\chi_i^{\mu}$. Therefore, we obtain, using first the centrality of $\rho_{n+1}$,
$$
\begin{array}{rcl}
\rho_{n+1}(xe_n)&=&\rho_{n+1}(e_nx)\\[0.4em]
&=&\rho^{\mu^{[b]}}\circ\operatorname{Tr}_{\Mat_{m_{\mu^{[b]}}}}\circ\Psi_{\mu^{[b]}}(E_{\chi_{i_b}^{\mu^{[b]}}}\tg_{w})\\[1em]
&=&\left\{\begin{array}{ll}
0 & \text{if $w(\chi_i^{\mu})\neq\chi_i^{\mu}$\,;}\\[0.3em]
\rho^{\mu^{[b]}}(\tT_{\pi_{i_b}^{-1} w\pi^{}_{i_b}})\ \  & \text{if $w(\chi_i^{\mu})=\chi_i^{\mu}$\,.}
\end{array}\right.  \end{array}$$
As we have $T_{\mu_1+\dots+\mu_b}^{-1}=u^{-2}(T_{\mu_1+\dots+\mu_b}-v)$ in $\cH^{\mu^{[b]}}$, we conclude that
\begin{equation}\label{eq-rho3}
\rho_{n+1}(xg_n^{-1})=\left\{\begin{array}{ll}
0 & \text{if $w(\chi_i^{\mu})\neq\chi_i^{\mu}$\,;}\\[0.3em]
\rho^{\mu^{[b]}}(\tT_{\pi_{i_b}^{-1} w \pi^{}_{i_b}}T_{\mu_1+\dots+\mu_b}^{-1})\ \  & \text{if $w(\chi_i^{\mu})=\chi_i^{\mu}$\,.}
\end{array}\right.
\end{equation}

To sum up, in (\ref{eq-rho1})--(\ref{eq-rho3}), we obtained first that
\[\rho_n(x)=\rho_{n+1}(xg_n)=\rho_{n+1}(xg_n^{-1})=0\,, \ \ \ \ \text{if $w(\chi_i^{\mu})\neq\chi_i^{\mu}$\,.} \]
Furthermore, if $w(\chi_i^{\mu})=\chi_i^{\mu}$, then we write $\tT_{\pi_i^{-1} w \pi^{}_i}=x_1\otimes\dots\otimes x_d\in\cH^{\mu}$, and we note that, due to (\ref{inc-Hmu}) and (\ref{form-pi}), $\tT_{\pi_{i_b}^{-1} w \pi^{}_{i_b}}$ is the image in $\cH^{\mu^{[b]}}$ of $\tT_{\pi_i^{-1} w \pi^{}_i}$ under the natural inclusion $\cH^{\mu}\subset\cH^{\mu^{[b]}}$. So, if $w(\chi_i^{\mu})=\chi_i^{\mu}$, Formulas (\ref{eq-rho1})--(\ref{eq-rho3}) read
\[\rho_n(x)=\rho^{\mu}(x_1\otimes\dots\otimes x_b \otimes\dots\otimes x_d)\ \ \ \ \ \text{and}\ \ \ \ \ \rho_{n+1}(xg_n^{\pm1})=\rho^{\mu^{[b]}}(x_1\otimes\dots\otimes x_bT_{\mu_b}^{\pm1} \otimes\dots\otimes x_d)\ .\]
We conclude the proof of the Lemma by noticing that, when $i$ runs through $\{1,\dots,m_{\mu}\}$, every $b$ such that $\mu_b\geq1$ is obtained, and moreover every element of $\cH^{\mu}$ can be written as $\tT_{\pi_i^{-1} w \pi^{}_i}$ for some $w\in\mS_n$ satisfying $w(\chi_i^{\mu})=\chi_i^{\mu}$.
\end{proof}

\begin{lemma}\label{lemme2}
The family $\{\rho_n\}_{n\geq1}$ satisfies the Markov condition (M2) if and only if the associated traces $\rho^{\mu}$ satisfy, for any $\mu\in\bigcup_{n\geq1}\Comp_d(n)$ and any $a\in\{1,\dots,d\}$ such that $\mu_a\geq1$,
\begin{equation}\label{mark-asso2}
\begin{array}{rcl}
 \rho^{\mu^{[a]}}(x_1\otimes\dots\otimes x_a T_k\otimes\dots\otimes x_d)&=&\rho^{\mu^{[a]}}(x_1\otimes\dots\otimes x_a T_k^{-1}\otimes\dots\otimes x_d)\\[0.4em]
  &=&\rho^{\mu}(x_1\otimes\dots\otimes x_a \otimes\dots\otimes x_d)\ ,\end{array}
\end{equation}
for any $k\in\{1,\dots,\mu_a\}$ and any $x_1\otimes\dots\otimes x_d\in\cH^{\mu}$ such that $x_a\in\cH_k\subset\cH_{\mu_a}$.
\end{lemma}
\begin{proof}
The "if" is a direct consequence of Lemma \ref{lemme1}, using the assumption with $k=\mu_a$.

To prove the "only if", we assume that the family $\{\rho_n\}_{n\geq1}$ satisfies the Markov condition (M2), and we proceed by induction on $\mu_a-k$ (it is very similar to the proof of Lemma \ref{lemme-tau-H}(i), so we only sketch it). The case $\mu_a-k=0$ is Lemma \ref{lemme1}. So let $k<\mu_a$. Then $T_{k+1}$ exists in $\cH_{\mu_a+1}$ and commutes with $x_a$. By centrality of $\rho^{\mu^{[a]}}$, we have
\[\rho^{\mu^{[a]}}(x_1\otimes\dots\otimes x_a T_k^{\pm1}\otimes\dots\otimes x_d)=\rho^{\mu^{[a]}}(x_1\otimes\dots\otimes x_a T_{k+1}T_k^{\pm1}T_{k+1}^{-1}\otimes\dots\otimes x_d)\ .\]
Using $T_{k+1}T_k^{\pm1}T_{k+1}^{-1}=T_{k}^{-1}T_{k+1}^{\pm1}T_{k}$ and the induction hypothesis, we obtain Formula (\ref{mark-asso2}).\end{proof}

\paragraph{Proof of Theorem \ref{theo-mark}.} We are now ready to prove Theorem \ref{theo-mark}.
Let $\{\rho_n\}_{n\geq1}$ be a Markov trace on  $\{Y_{d,n}\}_{n\geq1}$. As explained before Theorem \ref{theo-mark}, the existence of  associated traces $\rho^{\mu}$, such that Formula (\ref{rho-n}) holds, follows from the trace condition for $\rho_n$. We set $\alpha_{\mu}:=\rho^{\mu}(1)$ for any $\mu\in\Comp_d^0$.

Let $n\geq1$, $\mu\in\Comp_d(n)$ and $x=x_1\otimes\dots\otimes x_d\in\cH^{\mu}$. We will prove that
\begin{equation}\label{form-proof}
\rho^{\mu}(x)=\alpha_{[\mu]}\cdot \tau_{\mu_1}(x_1)\dots\tau_{\mu_d}(x_d)\ .
\end{equation}

First assume that $\mu=[\mu]$ (which is always true if $n=1$), so that every $\mu_a$ is 0 or 1. Then we have $x=1$ and Formula (\ref{form-proof}) follows from $\tau_1(1)=\tau_0(1)=1$.

\vskip .1cm
Assume now that $\mu\neq[\mu]$. We proceed by induction on $n$. First let $x\neq 1$, so that we have $a\in\{1,\dots,d\}$ such that $x_a\neq 1$ (in particular, $\mu_a\geq2$). 
We set $x_a=h_1T_kh_2$, where $k\in\{1,\dots,\mu_a-1\}$ and $h_1,h_2\in\cH_k\subset\cH_{\mu_a}$. We denote $\nu:=\mu_{[a]}\in\Comp_d(n-1)$ (that is, $\nu_a=\mu_a-1$ and $\nu_b=\mu_b$ if $b\neq a$), so that we have
\[\begin{array}{ll}\rho^{\mu}(x) & =\rho^{\mu}(x_1\otimes\dots\otimes h_1T_kh_2\otimes\dots\otimes x_d)\\[0.4em]
& =\rho^{\nu}(x_1\otimes\dots\otimes h_1h_2\otimes\dots\otimes x_d)\\[0.4em]
& =\alpha_{[\nu]}\tau_{\mu_1}(x_1)\dots\tau_{\mu_a-1}(h_1h_2)\dots\tau_{\mu_d}(x_d)\\[0.4em]
& =\alpha_{[\mu]}\tau_{\mu_1}(x_1)\dots\tau_{\mu_a}(h_1T_kh_2)\dots\tau_{\mu_d}(x_d)\ ,
\end{array}\]
where we first used (\ref{mark-asso2}) from Lemma \ref{lemme2}, then the induction hypothesis and finally the property of $\tau_{\mu_a}$ stated in item (i) of Lemma \ref{lemme-tau-H} (we also noted that $[\nu]=[\mu]$ since $\mu_a\geq2$).

\vskip .1cm
Finally let $x=1$. As $\mu\neq [\mu]$, we can choose $a\in\{1,\dots,d\}$ such that $\mu_a\geq 2$. We recall that, in $\cH_{\mu_a}$, we have $1=v^{-1}(T_1-u^2T_1^{-1})$. Setting again $\nu:=\mu_{[a]}\in\Comp_d(n-1)$, we calculate (below $v^{-1}(T_1-u^2T_1^{-1})$ is inserted in the $a$-th factor of the tensor product):
\[\rho^{\mu}(1)=\rho^{\mu}(1\otimes\dots1\otimes v^{-1}(T_1-u^2T_1^{-1})\otimes1\dots \otimes 1)=v^{-1}(1-u^2)\rho^{\nu}(1)\ ,\]
where we used (\ref{mark-asso2}) in Lemma \ref{lemme2}. Using the induction hypothesis, together with the fact that $[\nu]=[\mu]$ (since $\mu_a\geq2$), we obtain
\[\rho^{\mu}(1)=\alpha_{[\mu]}v^{-1}(1-u^2)\tau_{\mu_1}(1)\dots \tau_{\mu_a-1}(1)\dots\tau_{\mu_d}(1)=\alpha_{[\mu]}\tau_{\mu_1}(1)\dots \tau_{\mu_a}(1)\dots\tau_{\mu_d}(1)\ ,\]
where we used that $\tau_{\mu_a}(1)=v^{-1}(1-u^2) \tau_{\mu_a-1}(1)$ (Formula (\ref{tau1})). This concludes the proof of Formula (\ref{form-proof}).

\vskip .1cm
For the converse part of the theorem, we only have to check that, given a set of parameters $\{\alpha_{\mu}\,,\ \mu\in\Comp^0_d\}\subset \C[u^{\pm1},v^{\pm1}]$, the family $\{\rho_n\}_{n\geq1}$ of linear functions given by (\ref{rho-n}) and (\ref{rho-mu}) is a Markov trace. The trace condition is obviously satisfied as well as Equation (\ref{mark-asso}). The proof is concluded using Lemma \ref{lemme1}.\hfill$\square$

\begin{Rem}\label{rem-traces}
In view of Lemma \ref{lemme-tau-H}, item (ii), the associated traces of a Markov trace $\{\rho_n\}_{n\geq1}$ described by Theorem \ref{theo-mark} can be formally expressed as
\[\rho^{\mu}=\frac{\alpha_{[\mu]}}{\bigl(v^{-1}(1-u^2)\bigr)^{|[\mu]|-1}}\cdot\tau_n\,,\ \ \ \ \ \ \ \text{for any $\mu\in\Comp_d(n)$,}\]
where $\tau_n$ acts on $\cH^{\mu}$ by restriction from $\cH_n$ (note that Lemma \ref{lemme-tau-H}, item (ii), asserts in particular that the right hand side evaluated on $x\in\cH^{\mu}$ indeed belongs to $\C[u^{\pm1},v^{\pm1}]$).\hfill$\triangle$
\end{Rem}

\paragraph{Basis of the space of Markov traces.}
The classification of Markov traces on $\{Y_{d,n}\}_{n\geq1}$ given by Theorem \ref{theo-mark} can be formulated by saying that the space of Markov traces is a $\C[u^{\pm1},v^{\pm1}]$-module (for pointwise addition and scalar multiplication), which is free and of rank the cardinal of the set $\Comp_d^0$. We have
\[\sharp\Comp^0_d=\sum_{1\leq k\leq d}\left(\begin{array}{c}d\\ k\end{array}\right) = 2^d-1\ .\]
Further, Theorem \ref{theo-mark} provides a natural basis for this module. Indeed, for any $\mu^0\in\Comp^0_d$, let $\{\rho_{\mu^0,n}\}_{n\geq1}$
be the family of linear functions given by Formulas (\ref{rho-n})-(\ref{rho-mu}), for the following choice of parameters:
\[\alpha_{\mu^0}:=1\ \ \ \ \qquad\text{and}\qquad\ \ \ \ \alpha_{\nu^0}:=0\,,\ \ \text{for $\nu^0\in\Comp_d^0$ such that $\nu^0\neq\mu^0$\ .}\]
Then, $\{\rho_{\mu^0,n}\}_{n\geq1}$ is a Markov trace on $\{Y_{d,n}\}_{n\geq1}$ and it is given by
\begin{equation}\label{rho-mu0}
\rho_{\mu^0,n}(x)=\sum_{\underset{[\mu]=\mu^0}{\mu\in\Comp_d(n)}}(\tau_{\mu_1}\otimes\dots\otimes \tau_{\mu_d})\circ\operatorname{Tr}_{\Mat_{m_{\mu}}}\circ\Psi_{\mu}(E_{\mu}x)\,,\ \ \ \ \quad\text{for $n\geq1$ and $x\in Y_{d,n}$.}
\end{equation}
It follows from the classification that the following set is a $\C[u^{\pm1},v^{\pm1}]$-basis of the space of Markov traces on $\{Y_{d,n}\}_{n\geq1}$:
\begin{equation}\label{base-rho}
\left\{\,\{\rho_{\mu^0,n}\}_{n\geq1}\,\ |\ \ \mu^0\in\Comp^0_d\ \right\}\ .
\end{equation}

\section{Invariants for links and $\Z/d\Z$-framed links}\label{sec-inv}

Now that we have obtained a complete description of the Markov traces for $Y_{d,n}$, we will use them to deduce invariants for both framed and classical knots and links. In addition, we compare these invariants with the one coming from the study of the Iwahori-Hecke algebra of type $A$ : the HOMFLYPT polynomial.

\subsection{Classical braid group and HOMFLYPT polynomial}\label{subsec-inv-cl}

Let $n\in\Z_{\geq1}$. The braid group $B_n$ (of type $A_{n-1}$) is generated by elements $\sigma_1,\ldots, \sigma_{n-1},$
with defining relations:
\begin{equation}\label{def-B}
\begin{array}{rclcl}
\sigma_i\sigma_j & = & \sigma_j\sigma_i && \mbox{for all $i,j=1,\ldots,n-1$ such that $\vert i-j\vert > 1$,}\\[0.1em]
\sigma_i\sigma_{i+1}\sigma_i & = & \sigma_{i+1}\sigma_i\sigma_{i+1} && \mbox{for  all $i=1,\ldots,n-2$.}\\[0.1em]
\end{array}
\end{equation}
With the presentation (\ref{def-H}), the Iwahori--Hecke algebra $\cH_n$ is a quotient of the group algebra over $\C[u^{\pm1},v^{\pm1}]$ of the braid group $B_n$. We denote by $\delta_{\mathcal{H},n}$ the associated surjective morphism:
\[\delta_{\cH,n}\ :\ \C[u^{\pm1},v^{\pm1}]\bigl[B_n\bigr]\to\cH_n\,,\ \ \ \ \sigma_i\mapsto T_i\,,\ \ i=1,\dots,n-1\,.\]
The classical Alexander's theorem asserts that any link can be obtained as the closure of some braid. Next, the classical Markov's theorem gives necessary and sufficient conditions for two braids to have the same closure up to isotopy (see, \emph{e.g.}, \cite{jo}). The condition is that the two braids are equivalent under the equivalence relation generated by the conjugation and the so-called Markov move, namely, generated by
\begin{equation}\label{markov-move}\alpha\beta\sim\beta\alpha\ \ (\alpha,\beta\in B_n,\ n\geq1)\ \ \ \ \quad\text{and}\quad\ \ \ \ \alpha\sigma_n^{\pm1}\sim\alpha\ \ (\alpha\in B_n,\ n\geq1)\ .
\end{equation}
 Note that, in the Markov move, we consider $\alpha$ alternatively as an element of $B_n$ or of $B_{n+1}$ by the natural embedding $B_n\subset B_{n+1}$.

The conditions (M1) and (M2) in (\ref{Markov-tau}) for the Markov trace $\{\tau_n\}_{n\geq1}$ on the algebras $\cH_n$ reflect this equivalence relation and, as a consequence,  we obtain an isotopy invariant for links as follows. Let $K$ be a link and $\beta_K\in B_n$ a braid on $n$ strands having $K$ as its closure. The map $\Gamma_\cH$ from the set of links to the ring $\mathbb{C}[u^{\pm1},v^{\pm1}]$ defined by
\[\Gamma_\cH(K)=\tau_n\circ\delta_{\cH,n}(\beta_K)\ ,\]
only depends on the isotopy class of $K$ (this is immediate, comparing (\ref{Markov-tau}) and (\ref{markov-move})), and thus provides an isotopy invariants for links.

\begin{Rem}\label{Rem-Mar1}
The Laurent polynomial $\Gamma_\cH(K)$ in $u,v$ is called the HOMFLYPT polynomial. It was first obtained by a slightly different approach, using the Ocneanu trace on $\cH_n$ and a rescaling procedure, see \cite{jo} and references therein. We followed the approach in \cite[section 4.5]{GP}, where the connections between both approaches are specified. We will carefully detail this connection in the more general context of the Yokonuma--Hecke algebras below.\hfill$\triangle$
\end{Rem}

\subsection{$\Z/d\Z$-framed braid group and $\Z/d\Z$-framed links}\label{subsec-FB}

Roughly speaking, a $\Z/d\Z$-framed braid is a usual braid with an element of $\Z/d\Z$ (the \emph{framing}) attached to each strand. Similarly, a $\Z/d\Z$-framed link is a classical link where each connected component carries a framing in $\Z/d\Z$. The notion of isotopy for framed links is generalized straightforwardly from the classical setting. We refer to \cite{jula3,kosm} for more details on framed braids and framed links. 

\vskip .1cm
Let $d\in\Z_{\geq1}$. The $\Z/d\Z$-framed braid group, denoted by $\Z/d\Z\wr B_n$, is (isomorphic to) the semi-direct product of the abelian group $\left(\Z/d\Z\right)^n$ by the braid group $B_n$, where the action of $B_n$ on $\left(\Z/d\Z\right)^n$ is by permutation. In other words, the group $\Z/d\Z\wr B_n$ is generated by elements $\sigma_1,\sigma_2,\ldots,\sigma_{n-1}, t_1,\ldots, t_n,$
and relations:
\begin{equation}\label{def-FB}\begin{array}{rclcl}
\sigma_i\sigma_j & = & \sigma_j\sigma_i && \mbox{for all $i,j=1,\ldots,n-1$ such that $\vert i-j\vert > 1$,}\\[0.1em]
\sigma_i\sigma_{i+1}\sigma_i & = & \sigma_{i+1}\sigma_i\sigma_{i+1} && \mbox{for  all $i=1,\ldots,n-2$,}\\[0.1em]
t_it_j & =  & t_jt_i &&  \mbox{for all $i,j=1,\ldots,n$,}\\[0.1em]
\sigma_it_j & = & t_{s_i(j)}\sigma_i && \mbox{for all $i=1,\ldots,n-1$ and $j=1,\ldots,n$,}\\[0.1em]
t_j^d   & =  &  1 && \mbox{for all $j=1,\ldots,n$.}
\end{array}
\end{equation}
The closure of a $\Z/d\Z$-framed braid is naturally a $\Z/d\Z$-framed link (the framing of a connected component is the sum of the framings of the strands forming this component after closure). Given a classical link, from the classical Alexander's theorem, we have a classical braid closing to this link, and it is immediate that by adding a suitable framing on this braid, one can obtain any possible framing on the given link. So the analogue of Alexander's theorem is also true for $\Z/d\Z$-framed braids and links. 

\vskip .1cm
Moreover, the Markov's theorem has also been generalized to the $\Z/d\Z$-framed setting (see \cite[Lemma 1]{kosm} or \cite[Theorem 6]{jula3}).
The necessary and sufficient conditions for two $\Z/d\Z$-framed braids to have the same closure up to isotopy is formally the same as for usual braids; namely, the two braids have to be equivalent under the equivalence relation generated by
\begin{equation}\label{markov-move2}
\tilde{\alpha}\tilde{\beta}\sim\tilde{\beta}\tilde{\alpha}\ \ (\tilde{\alpha},\tilde{\beta}\in \Z/d\Z\wr B_n,\ n\geq1)\ \ \ \ \quad\text{and}\quad\ \ \ \ \tilde{\alpha}\sigma_n^{\pm1}\sim\tilde{\alpha}\ \ (\tilde{\alpha}\in \Z/d\Z\wr B_n,\ n\geq1)\ .
\end{equation}
The conditions (M1) and (M2) in (\ref{Markov-rho}) for a Markov trace $\{\rho_n\}_{n\geq1}$ on the Yokonuma--Hecke algebras reflect this equivalence relation, and this will allow to use the Markov traces obtained in the previous section to construct isotopy invariants for $\Z/d\Z$-framed links.

\paragraph{A family of morphisms from the group algebra of $\Z/d\Z\wr B_n$ to $Y_{d,n}$.} Let $\gamma$ be another indeterminate and set $R:=\C[u^{\pm1},v^{\pm1},\gamma^{\pm1}]$.
We define:
\begin{equation}\label{def-delta}
\delta^{\gamma}_{Y,n}\ :\ \ \sigma_i\mapsto \bigl(\gamma+(1-\gamma)e_i\bigr)g_i\ \ (i=1,\dots,n-1)\,,\ \ \ \ \ t_j\mapsto t_j\ \ (j=1,\dots,n)\,.
\end{equation}

\begin{lemma}\label{lem-delta}
The map $\delta^{\gamma}_{Y,n}$ extends to an algebras homomorphism from $R\bigl[\Z/d\Z\wr B_n\bigr]$ to $RY_{d,n}$.
\end{lemma}
\begin{proof}
We have to check that the defining relations (\ref{def-FB}) are satisfied by the images of the generators, and also that the images of the generators are invertible elements of $RY_{d,n}$. For the latter statement, it is easily checked that
\begin{equation}\label{delta-inv}
\Bigl(\bigl(\gamma+(1-\gamma)e_i\bigr)g_i\Bigr)^{-1}=\bigl(\gamma^{-1}+(1-\gamma^{-1})e_i\bigr)g_i^{-1}\,,\ \ \ \ \ i=1,\dots,n-1.
\end{equation}
The three last relations in (\ref{def-FB}) are satisfied since the elements $e_i$'s and $t_j$'s commute. Then, if $|i-j|>1$, a direct calculation shows that the image of the first relation in (\ref{def-FB}) is
\[(\gamma+(1-\gamma)e_i\bigr)(\gamma+(1-\gamma)e_j\bigr)g_ig_j= (\gamma+(1-\gamma)e_j\bigr)(\gamma+(1-\gamma)e_i\bigr)g_jg_i\ ,\]
since $g_ie_j=e_jg_i$ and $g_je_i=e_ig_j$ whenever $|i-j|>1$. This relation is satisfied in $RY_{d,n}$.

Finally, using again the commutation relations between the generators $g_i$'s and $t_j$, we calculate the image of the first relation in (\ref{def-FB}), and obtain
\[(\gamma+(1-\gamma)e_i\bigr)(\gamma+(1-\gamma)e_{i,i+2}\bigr)(\gamma+(1-\gamma)e_{i+1}\bigr)\bigl(g_ig_{i+1}g_i-g_{i+1}g_ig_{i+1})=0\ ,\]
where $e_{i,i+2}:=\displaystyle \frac{1}{d}\sum_{0\leq s \leq d-1}t_i^st_{i+2}^{-s}$. This relation is also satisfied in $RY_{d,n}$.
\end{proof}

\begin{Rem}
If we specialize the parameter $\gamma$ to $1$, we obtain the natural surjective morphism from the group algebra over $\C[u^{\pm1},v^{\pm1}]$ of the group $\Z/d\Z\wr B_n$ to its quotient $Y_{d,n}$ with the presentation (\ref{rel-def-Y}) of Section \ref{sec-def}. For other values of $\gamma$, this morphism is related with other equivalent presentations of $Y_{d,n}$; see below Subsection \ref{subsec-jula}.
\end{Rem}

\subsection{Invariants for classical and $\Z/d\Z$-framed links from $Y_{d,n}$}\label{subsec-inv-F}

\paragraph{Invariants for $\Z/d\Z$-framed links from $Y_{d,n}$.}
Let $\{\rho_n\}_{n\geq1}$ be a Markov trace on the Yokonuma--Hecke algebras $\{Y_{d,n}\}_{n\geq1}$ (defined by Conditions (M1) and (M2) in (\ref{Markov-rho})), and extend it $R$-linearly to $\{RY_{d,n}\}_{n\geq1}$. 

Let $\tilde{K}$ be a $\Z/d\Z$-framed link and $\tilde{\beta}_{\tilde{K}}\in \Z/d\Z\wr B_n$ a $\Z/d\Z$-framed braid on $n$ strands having $\tilde{K}$ as its closure. We define a map $\FG^{\gamma}_{Y,\rho}$ from the set of $\Z/d\Z$-framed links to the ring $R$ by
\begin{equation}\label{tgamma}
\FG^{\gamma}_{Y,\rho}(\tilde{K})=\rho_n\circ\delta^{\gamma}_{Y,n}(\tilde{\beta}_{\tilde{K}})\ ,
\end{equation}
where the map $\delta^{\gamma}_{Y,n}$ is defined in (\ref{def-delta}). For $\mu^0\in\Comp^0_d$, we denote by $\FG^{\gamma}_{Y,\mu_0}$ the map corresponding to the Markov trace $\{\rho_{\mu^0,n}\}_{n\geq1}$ considered in (\ref{rho-mu0}). The classification of Markov traces of the previous section, together with the construction detailed in this section, lead to the following result.
\begin{Th}\label{theo-inv}
\begin{enumerate}
\item[\bf{1.}] For any Markov trace $\{\rho_n\}_{n\geq1}$ on $\{Y_{d,n}\}_{n\geq1}$, the map $\FG^{\gamma}_{Y,\rho}$ is an isotopy invariant for $\Z/d\Z$-framed links with values in $R=\C[u^{\pm1},v^{\pm1},\gamma^{\pm1}]$.

\item[\bf{2.}] The set of invariants for $\Z/d\Z$-framed links obtained from the Yokonuma--Hecke algebras via this construction consists of  all $R$-linear combinations of invariants from the set
\begin{equation}\label{basis-inv}
\left\{\ \FG^{\gamma}_{Y,\mu^0}\ \ |\ \ \mu^0\in\Comp^0_d\ \right\}\ .
\end{equation}
\end{enumerate}
\end{Th}
\begin{proof} \textbf{1.} By the Markov's theorem for $\Z/d\Z$-framed links, we have to check that the map given by  $\{\rho_n\circ\delta^{\gamma}_{Y,n}\}_{n\geq1}$ from the set of $\Z/d\Z$-framed braids to the ring $R$ coincide on equivalent braids, for the equivalence relation generated by the moves in (\ref{markov-move2}). 

From the trace condition (M1) in (\ref{Markov-rho}) for $\{\rho_n\}_{n\geq1}$ and the fact that the maps $\delta^{\gamma}_{Y,n}$ ($n\geq1$) are algebra morphisms (Lemma \ref{lem-delta}), it follows at once that the map $\rho_n\circ\delta^{\gamma}_{Y,n}$ ($n\geq1$) coincides on $\tilde{\alpha}\tilde{\beta}$ and $\tilde{\beta}\tilde{\alpha}$ for any two braids $\tilde{\alpha},\tilde{\beta}\in \Z/d\Z\wr B_n$.

Next, let $n\geq1$ and $\tilde{\alpha}\in \Z/d\Z\wr B_n$. We set $x_{\tilde{\alpha}}:=\delta^{\gamma}_{Y,n}(\tilde{\alpha})\in Y_{d,n}$. Note that, seeing $x_{\tilde{\alpha}}$ as an element of $Y_{d,n+1}$, we also have $x_{\tilde{\alpha}}:=\delta^{\gamma}_{Y,n+1}(\tilde{\alpha})$, by definition of $\{\delta^{\gamma}_{Y,n}\}_{n\geq1}$. From (\ref{def-delta}) and (\ref{delta-inv}), we have
\[\rho_{n+1}\circ\delta^{\gamma}_{Y,n+1}(\tilde{\alpha}\sigma_n^{\pm1})=\rho_{n+1}\Bigl(x_{\tilde{\alpha}}\bigl(\gamma^{\pm1}+(1-\gamma^{\pm1})e_n\bigr)g_n^{\pm1}\Bigr)\ ,\]
and we need to prove that it is equal to $\rho_{n}\circ\delta^{\gamma}_{Y,n}(\tilde{\alpha})=\rho_{n}(x_{\tilde{\alpha}})$. This will follow from the Markov condition (M2) in (\ref{Markov-rho}) for $\{\rho_n\}_{n\geq1}$ together with the following fact:
\begin{equation}\label{eq-proof}
\rho_{n+1}(xe_ng^{\pm1}_n)=\rho_{n+1}(xg^{\pm1}_n)\,,\ \ \ \ \ \text{for any $n\geq1$ and $x\in Y_{d,n}$\,.}
\end{equation}
This last relation is true for any linear map $\kappa$ on $Y_{d,n+1}$ satisfying the trace condition since, for $s\in\{1,\dots,d\}$,
\[\kappa(xt_{n+1}^{-s}t_n^{s}g^{\pm1}_n)=\kappa(xt_{n+1}^{-s}g^{\pm1}_nt_{n+1}^{s})=\kappa(t_{n+1}^{s}xt_{n+1}^{-s}g^{\pm1}_n)=\kappa(xg^{\pm1}_n)\ ,\]
where we used successively the relation $t_ng^{\pm1}_n=g^{\pm1}_nt_{n+1}$, the trace condition and the fact that $t_{n+1}$ commutes with $x\in Y_{d,n}$.

\vskip .1cm
\textbf{2.} This is simply a reformulation of the classification of Markov traces $\{\rho_n\}_{n\geq1}$ on $\{Y_{d,n}\}_{n\geq1}$ given by Theorem \ref{theo-mark} leading to the basis $\left\{\,\{\rho_{\mu^0,n}\}_{n\geq1}\,\ |\ \ \mu^0\in\Comp^0_d\ \right\}$ in (\ref{base-rho}).
\end{proof}

\paragraph{Invariants for classical links from $Y_{d,n}$.} The classical braid group $B_n$ is naturally a subgroup of the $\Z/d\Z$-framed braid group $\Z/d\Z\wr B_n$ (a classical braid is seen as a $\Z/d\Z$-framed braid with all framings equal to 0). Therefore, one can restrict the maps $\FG^{\gamma}_{Y,\rho}$ in (\ref{tgamma}) to classical links, and obtain invariants for classical links since the Markov's theorem is formally the same for classical and $\Z/d\Z$-framed links; compare (\ref{markov-move}) and (\ref{markov-move2}).

Explicitly, let $K$ be a link and $\beta_{K}\in B_n$ a braid on $n$ strands having $K$ as its closure. We now see $\beta_K$ as an element of the $\Z/d\Z$-framed braid group $\Z/d\Z\wr B_n$ and we set
\begin{equation}\label{gamma}
\Gamma^{\gamma}_{Y,\rho}(K)=\rho_n\circ\delta^{\gamma}_{Y,n}(\beta_{K})\ .
\end{equation}
For $\mu^0\in\Comp^0_d$, we denote by $\Gamma^{\gamma}_{Y,\mu^0}$ the map corresponding to the Markov trace $\{\rho_{\mu^0,n}\}_{n\geq1}$ considered in (\ref{rho-mu0}). According to the above discussion, the following corollary is immediately deduced from Theorem \ref{theo-inv}
\begin{Cor}\label{coro-inv2}
\begin{enumerate}
\item[\bf{1.}] For any Markov trace $\{\rho_n\}_{n\geq1}$ on $\{Y_{d,n}\}_{n\geq1}$, the map $\Gamma^{\gamma}_{Y,\rho}$ is an isotopy invariant for classical links with values in $R=\C[u^{\pm1},v^{\pm1},\gamma^{\pm1}]$.

\item[\bf{2.}] The set of invariants for classical links obtained from the Yokonuma--Hecke algebras via this construction consists of  all $R$-linear combinations of invariants from the set
\begin{equation}\label{basis-inv2}
\left\{\ \Gamma^{\gamma}_{Y,\mu^0}\ \ |\ \ \mu^0\in\Comp^0_d\ \right\}\ .
\end{equation}
\end{enumerate}
\end{Cor}
Note that, in the definition of the invariant $\Gamma^{\gamma}_{Y,\rho}(K)$ in (\ref{gamma}), even though the word $\beta_K$ only contains generators $\sigma_i$ (and no $t_j$), the image $\delta^{\gamma}_{Y,n}(\beta_{K})$ in the algebra	 $Y_{d,n}$ does involve in general the generators $t_j$ (more precisely, it involves the elements $e_i$). Indeed, first, the image of $\sigma_i$ by the map $\delta^{\gamma}_{Y,n}$ contains the idempotent $e_i$. Besides, even if $\gamma$ is specialized to $1$, as soon as one $\sigma_i^2$ for example appears in $\beta_K$, then the last relation of (\ref{rel-def-Y}) is used to calculate $\rho_n\circ\delta^1_{Y,n}(\beta_{K})$, and this last relation involves the idempotent $e_i$.

\begin{Exa}\label{exa2}
Let $d=2$. We will explicitly give $\Gamma^{\gamma}_{Y,\mu^0}(K)$ for $\mu^0\in\Comp_2^0$ and some classical links $K$. Using the notations of \cite{KAT}, let $K_1=\text{L10a46}$ and $K_2=\text{L10a110}$.
For each of these two links, one can find in \cite{KAT} a braid on $4$ strands closing to the link. Namely, the braid $\beta_1=\sigma_1^2\sigma_2^{-1}\sigma_3^{-1}\sigma_2^{-1}\sigma_1^3\sigma_2^{-1}\sigma_3^{-1}\sigma_2^{-1}\sigma_1$ admits $K_1$ as its closure, while the braid $\beta_1=\sigma_1^{-1}\sigma_2^{3}\sigma_1^{-1}\sigma_3^{-1}\sigma_2^{3}\sigma_3^{-1}$ admits $K_2$ as its closure. Thus we can use the calculations made in Example \ref{exa1}.

\vskip .2cm
\noindent $\bullet$
We first consider $\mu^0=(1,0)$ or $\mu^0=(0,1)$. Then we have, by definition, $\Gamma^{\gamma}_{Y,\mu^0}(K_i)=\tau_4\circ\Psi_{\mu}\circ\delta^{\gamma}_{Y,4}(\beta_i)$ ($i=1,2$), where $\mu$ is the composition $(4,0)$ or $(0,4)$, and $\tau_4$ comes from the unique Markov trace $\{\tau_n\}_{n\geq1}$ on the Iwahori--Hecke algebras. It is straightforward to see that in this situation, from the formulas in Example \ref{exa1}, we have $\delta^{\gamma}_{Y,4}(\beta_i)=\delta_{H,4}(\beta_i)$, and in turn that we have $\Gamma^{\gamma}_{Y,\mu^0}(K_i)=\Gamma_{\cH}(K_i)$  (the HOMFLYPT polynomial). This is a general property of the invariants $\Gamma^{\gamma}_{Y,\mu^0}$ when $|\mu^0|=1$; see Proposition \ref{prop-comp1} below.

\vskip .2cm
\noindent $\bullet$
Then we consider $\mu^0=(1,1)$.  By definition, we have ($i=1,2$)
\begin{equation}\label{11}
\Gamma^{\gamma}_{Y,\mu^0}(K_i)=\Bigl((\tau_3\otimes\tau_1)\circ\text{Tr}\circ\Psi_{(3,1)}+(\tau_1\otimes\tau_3)\circ\text{Tr}\circ\Psi_{(1,3)}+(\tau_2\otimes\tau_2)\circ\text{Tr}\circ\Psi_{(2,2)}\Bigr)\circ\delta^{\gamma}_{Y,4}(\beta_i),
\end{equation}
where $\text{Tr}$ is the usual trace of a matrix. So we need first to calculate $\Psi_{\mu}\circ\delta^{\gamma}_{Y,4}(\beta_{i})$ for $\mu=(3,1),(1,3),(2,2)$. Take for example $i=1$ and $\mu=(3,1)$. According to Example \ref{exa1}, the generators $g_1$, $g_2$ and $g_3$ map under $\Psi_{(3,1)}\circ\delta^{\gamma}_{Y,4}$ respectively to
\[\left(\begin{array}{cccc}\cdot & u\gamma & \cdot & \cdot\\ u\gamma & \cdot & \cdot & \cdot\\ \cdot & \cdot & T_1 & \cdot\\ \cdot & \cdot & \cdot & T_1\end{array}\right)\,,\ \ \ \ \ 
\left(\begin{array}{cccc}T_1 & \cdot & \cdot & \cdot\\ \cdot & \cdot & u\gamma & \cdot\\ \cdot & u\gamma & \cdot & \cdot\\ \cdot & \cdot & \cdot & T_2\end{array}\right)\ \ \ \ \ \text{and}\ \ \ \ \ 
\left(\begin{array}{cccc}T_2 & \cdot & \cdot & \cdot\\ \cdot & T_2 & \cdot & \cdot\\ \cdot & \cdot & \cdot & u\gamma\\ \cdot & \cdot & u\gamma & \cdot\end{array}\right)\,.\]
Performing the matrix multiplication corresponding to the braid $\beta_1$ given above, we obtain
\[\Psi_{(3,1)}\circ\delta^{\gamma}_{Y,4}(\beta_{1})=\!\left(\begin{array}{cccc}\cdot & \cdot & (u\gamma)^2\, T_1^2T_2^{-1}T_1^{-1}T_2T_1^{-1} & \cdot\\ \cdot & \displaystyle\frac{T_1^2T_2^{-1}T_1^3T_2T_1}{(u\gamma)^{4}} & \cdot & \cdot\\ \cdot & \cdot & \cdot & u\gamma\,T_2^{-1}T_1^3T_2^{-1}\\ (u\gamma)^5\,T_1^{-1}T_2^{-1}T_1^{-1}T_2^{-1}T_1 & \cdot & \cdot & \cdot\end{array}\right)\]
This gives a contribution ({\it i.e.} a term in the sum (\ref{11})) to $\Gamma^{\gamma}_{Y,\mu^0}(K_1)$ equal to $\displaystyle\frac{1}{(u\gamma)^{4}}\tau_3(T_1^2T_2^{-1}T_1^3T_2T_1)$. It is easy to see that the composition $(1,3)$ gives the same contribution to $\Gamma^{\gamma}_{Y,\mu^0}(K_1)$. A similar calculation shows that the composition $(2,2)$ gives a contribution equal to 0 (this can also be deduced without calculation from the fact that the underlying permutation of $\beta_1$ is $(1,2,4)$ and this cycle structure makes impossible for $\Psi_{(2,2)}\circ\delta^{\gamma}_{Y,4}(\beta_1)$ to have a non-zero diagonal term).

A similar procedure for $\beta_2$ shows that the compositions $(3,1)$ and $(1,3)$ give both a contribution to $\Gamma^{\gamma}_{Y,\mu^0}(K_2)$ equal to 0, while the composition $(2,2)$ gives a contribution equal to $\displaystyle\frac{2}{(u\gamma)^{4}}\tau_2(T_1^3)^2$.

Quite remarkably, even though the two calculations involve different compositions and different elements of Iwahori--Hecke algebras, these two calculations lead finally to
\[\Gamma^{\gamma}_{Y,\mu^0}(K_1)=\Gamma^{\gamma}_{Y,\mu^0}(K_2)=\frac{2}{(u\gamma)^{4}}(2u^2-u^4+v^2)^2\ .\]
From \cite{KAT}, we note that these two links $K_1,K_2$ are topologically different and are however not distinguished by the HOMFLYPT polynomial. We just checked that they are not distinguished neither by the invariants $\Gamma^{\gamma}_{Y,\mu^0}$ (when $d=2$) coming from the Yokonuma--Hecke algebras.

In fact, we will prove below in Proposition \ref{prop-comp2} that (for any $d$), the set of invariants $\{\Gamma^{\gamma}_{Y,\mu^0}\}$ is topologically equivalent to the HOMFLYPT polynomial when restricted to classical knots. We note that the question remains open for classical links which are not knots. However, computational data seem to indicate that the invariants are topologically equivalent as well for all classical links (we checked this by direct calculations, as in the example here, for $d=2$ and links up to 10 crossings; see also \cite{chla}).  
\hfill$\triangle$
\end{Exa}

\subsection{Comparison of invariants for classical links}

As already explained, when calculating the invariant for a classical link using the Yokonuma--Hecke algebras, the additional generators $t_j$ play a non-trivial role, and therefore these invariants are \emph{a priori} different from the HOMFLYPT polynomial. In this part, we will compare the set of invariants for classical links obtained from the Yokonuma--Hecke algebras with the HOMFLYPT polynomial. The main question is whether or not they are topologically equivalent.

According to Corollary \ref{coro-inv2}, we can express any invariant for classical links obtained from the algebras $Y_{d,n}$ via the above construction as a linear combination of the invariants denoted $\Gamma^{\gamma}_{Y,\mu^0}$, where $\mu^0\in\Comp^0_d$. The main question is whether we have, for any two classical links $K_1,K_2$,
\[\Bigl(\ \forall\mu^0\in\Comp^0_d\,,\ \ \Gamma^{\gamma}_{Y,\mu^0}(K_1)=\Gamma^{\gamma}_{Y,\mu^0}(K_2)\ \Bigr)
\ \quad\stackrel{\textbf{?}}{\Longleftrightarrow}\quad\ \Gamma_\cH(K_1)=\Gamma_\cH(K_2)\ .\]

In this part, we will show first that the set $\{\Gamma^{\gamma}_{Y,\mu^0}\ |\ \mu^0\in\Comp^0_d\}$ contains $\Gamma_\cH$, so that only one half of the equivalence is not trivial. Secondly, we will show that this equivalence is true whenever we restrict our attention to classical knots. We refer to  \cite{vous,chla} for similar results about Juyumaya--Lambropoulou invariants. Note that these invariants  are shown     in Section \ref{subsec-jula} below to be certain linear combinations of the set   $\{\Gamma^{\gamma}_{Y,\mu^0}\ |\ \mu^0\in\Comp^0_d\}$ for some specific value of $\gamma$.  

\paragraph{The HOMFLYPT polynomial from $Y_{d,n}$.} Among the basic invariants $\Gamma^{\gamma}_{Y,\mu^0}$, we consider in this paragraph the ones for which $|\mu^0|=1$.

\begin{Prop}\label{prop-inv1}
Let $\mu^0\in\Comp_d(1)$. We have, for any classical link $K$,
\[\Gamma^{\gamma}_{Y,\mu^0}(K)=\Gamma_\cH(K)\ .\]
In particular, the set of invariants for classical links obtained from $Y_{d,n}$ contains the HOMFLYPT polynomial.
\end{Prop}
\begin{proof}
Let $\mu^0\in\Comp_d(1)$ (so that $\mu^0$ automatically belongs to $\Comp^0_d$). So there exists $a\in\{1,\dots,d\}$ such that $\mu^0=(0,\dots,0,1,0,\dots,0)$ with $1$ in $a$-th position. Note that a composition $\mu$ with $d$ parts satisfies $[\mu]=\mu^0$ if and only if $\mu=(0,\dots,0,n,0,\dots,0)$ with $n$ in $a$-th position for some $n\geq1$. 

So let $n\geq 1$ and $\mu=(0,\dots,0,n,0,\dots,0)$ with $n$ in $a$-th position. In this situation, we have $\cH^{\mu}\cong\cH_n$ and $m_{\mu}=1$. According to Formula (\ref{rho-mu0}),  the linear function $\rho_{\mu^0,n}$ is then given by
\begin{equation}\label{eq-prop-1}
\rho_{\mu^0,n}(x)=\tau_n\circ\Psi_{\mu}(E_{\mu}x)\,,\ \ \ \ \quad\text{for any $x\in Y_{d,n}$.}
\end{equation}
The defining formula (\ref{def-psi}) for the isomorphism $\Psi_{\mu}$ becomes simply $\Psi_{\mu}(E_{\mu}g_w)=T_w$ for $w\in\mS_n$, and in particular, we have
\begin{equation}\label{eq-prop-2}
\Psi_{\mu}(E_{\mu}g_i)=T_i\,,\ \ \ \ \quad\text{for any $i=1,\dots,n-1$.}
\end{equation}
Note that $E_{\mu}e_i=E_{\mu}$, for any $i=1,\dots,n-1$, since, for the considered $\mu$, we have $E_{\mu}t_i=E_{\mu}t_{i+1}$ (both are equals to $\xi_aE_{\mu}$). Therefore
\begin{equation}\label{eq-prop-3}
E_{\mu}\delta^{\gamma}_{Y,n}(\sigma_i)=E_{\mu}\bigl(\gamma+(1-\gamma)e_i\bigr)g_i=E_{\mu}g_i\,,\ \ \ \ \quad\text{for any $i=1,\dots,n-1$.}
\end{equation}
To conclude, let $\beta\in B_n$ be a classical braid. Equations (\ref{eq-prop-2})-(\ref{eq-prop-3}), together with the fact that $\Psi_{\mu}$ is a morphism, yields 
\[\Psi_{\mu}\bigl(\,E_{\mu}\delta^{\gamma}_{Y,n}(\beta)\,\bigr)=\delta_{\cH,n}(\beta)\ ,\]
which gives in turn, using (\ref{eq-prop-1}), that $\rho_{\mu^0,n}\circ\delta^{\gamma}_{Y,n}(\beta)=\tau_n\circ \delta_{\cH,n}(\beta)$. This is the desired result.
\end{proof}

\paragraph{Equivalence of invariants for classical knots.} Let $\beta\in B_n$, for some $n\geq1$, be a classical braid. From the presentation (\ref{def-B}) of $B_n$, there is a surjective morphism from $B_n$ to the symmetric group $\mS_n$ given by $\sigma_i\mapsto s_i=(i,i+1)$ for $i=1,\dots,n-1$. We will denote $\bar{\beta}\in\mS_n$ the image of $\beta$ and refer to it as the underlying permutation of $\beta$.

Now, the necessary and sufficient condition for the closure of $\beta$ to be a knot (that is, a link with only one connected component) is that the underlying permutation $\bar{\beta}$ leaves no non-trivial subset of $\{1,\dots,n\}$ invariant. In other words, the closure of $\beta$ is a knot if and only if $\bar{\beta}$ is a cycle of length $n$.

\begin{Prop}\label{prop-inv2}
For any classical knot $K$ and any $\mu^0\in\Comp^0_d$,
\[\Gamma^{\gamma}_{Y,\mu^0}(K)=\left\{\begin{array}{ll}\Gamma_\cH(K) & \text{if $|\mu^0|=1$,}\\[0.2em]
0 & \text{otherwise\ .}\end{array}\right.\]
In particular, for classical knots, the invariants obtained from $Y_{d,n}$ are topologically equivalent to the HOMFLYPT polynomial.
\end{Prop}
\begin{proof}
Let $K$ be a classical knot and $\beta\in B_n$, for some $n\geq1$, a classical braid closing to $K$. To save space during the proof, we set $x_{\beta}:=\delta^{\gamma}_{Y,n}(\beta)\in Y_{d,n}$. 

Let $\mu^0\in\Comp^0_d$ with $|\mu^0|>1$. According to Proposition \ref{prop-inv1}, we only have to prove that $\Gamma^{\gamma}_{Y,\mu^0}(K)=0$ which is equivalent to $\rho_{\mu^0,n}(x_{\beta})=0$. We will actually prove the following stronger statement:
\[\operatorname{Tr}_{\Mat_{m_{\mu}}}\circ\Psi_{\mu}\bigl(E_{\chi^{\mu}_k}x_{\beta}\bigr)=0\,,\ \ \ \ \qquad\text{for any $\mu\in\Comp_d(n)$ such that $[\mu]=\mu^0$, and any $k\in\{1,\dots,m_{\mu}\}$.}\]
The required assertion will then follow from (\ref{rho-mu0}) and the fact that $E_{\mu}x_{\beta}=\sum_{1\leq k\leq m_{\mu}}E_{\chi^{\mu}_k}x_{\beta}$.

\vskip .1cm
We first note that, in the framed braid group $\Z/d\Z\wr B_n$, we have $\beta t_j=t_{\bar{\beta}(j)}\beta$, for $j=1,\dots,n$, due to the fourth relation in (\ref{def-FB}). Therefore, in $Y_{d,n}$, we have $x_{\beta} t_j=t_{\bar{\beta}(j)}x_{\beta}$, for $j=1,\dots,n$, and in turn $x_{\beta} E_{\chi}=E_{\bar{\beta}(\chi)}x_{\beta}$ for any character $\chi$ of $(\Z/d\Z)^n$.

Then let $\mu\in\Comp_d(n)$ such that $[\mu]=\mu^0$ and $k\in\{1,\dots,m_{\mu}\}$. We recall that $\pi\in\mS_n$ satisfies $\pi(\chi^{\mu}_k)=\chi^{\mu}_k$ if and only if $\pi$ belongs to a subgroup of $\mS_n$ conjugated to $\mS^{\mu}$ (namely, to $\pi_{k,\mu}\mS^{\mu}\pi_{k,\mu}^{-1}$ with the notations of Section \ref{sec-def}). By the assumption on $\mu$, we have at least two integers $a,b\in\{1,\dots,d\}$ such that $\mu_a,\mu_b\geq1$, and thus the subgroup $\mS^{\mu}=\mS_{\mu_1}\times\dots\times\mS_{\mu_d}$ contains no cycle of length $n$. This means in particular that $\bar{\beta}^{-1}(\chi^{\mu}_k)\neq\chi^{\mu}_k$ since $\bar{\beta}$ is a cycle of length $n$ as $K$ is a knot.

Finally, we write $E_{\chi^{\mu}_k}x_{\beta}=E_{\chi^{\mu}_k}^2x_{\beta}=E_{\chi^{\mu}_k}x_{\beta}E_{\bar{\beta}^{-1}(\chi^{\mu}_k)}$ and we conclude the proof with the following calculation
\[\operatorname{Tr}_{\Mat_{m_{\mu}}}\circ\Psi_{\mu}(E_{\chi^{\mu}_k}x_{\beta}E_{\bar{\beta}^{-1}(\chi^{\mu}_k)})
=\operatorname{Tr}_{\Mat_{m_{\mu}}}\circ\Psi_{\mu}(E_{\bar{\beta}^{-1}(\chi^{\mu}_k)}E_{\chi^{\mu}_k}x_{\beta})
=0\ ,\]
where we used that $E_{\chi}E_{\chi'}=0$ if $\chi\neq\chi'$.
\end{proof}

\begin{Rem}
Note that in general, for an arbitrary Markov trace $\{\rho_n\}_{n\geq1}$, the invariant $\Gamma^{\gamma}_{Y,\rho}(K)$ is an element of the ring $R=\C[u^{\pm1},v^{\pm1},\gamma^{\pm1}]$. For a classical knot $K$, Proposition \ref{prop-inv2} asserts in particular that every invariant $\Gamma^{\gamma}_{Y,\mu^0}(K)$ ($\mu^0\in\Comp^0_d$) belongs actually to the subring $\C[u^{\pm1},v^{\pm1}]$. Further, for a classical link $K$, Proposition \ref{prop-inv1} asserts in particular that, when $\mu^0\in\Comp_d(1)$, the invariant $\Gamma^{\gamma}_{Y,\mu^0}(K)$ belongs as well to the subring $\C[u^{\pm1},v^{\pm1}]$. So for classical links, the parameter $\gamma$ in fact starts to play a non-trivial role when $K$ is not a knot and $|\mu^0|>1$ (see Example \ref{exa2}).
\end{Rem}

\subsection{Connections with the approach of Juyumaya--Lambropoulou}\label{subsec-jula}

\paragraph{Analogue of the Ocneanu trace and invariants.} Let $q$ be an indeterminate. In \cite{chla,ju2,jula0,jula1,jula2,jula3}, the Yokonuma--Hecke algebra is presented as a certain quotient of the group algebra over $\mathbb{C}[q,q^{-1}]$ of the $\Z/d\Z$-framed braid group $\Z/d\Z\wr B_n$. Namely, there are generators $G_1,G_2,\ldots,G_{n-1}$ and $t_1,\ldots, t_n,$ satisfying the same relations as in (\ref{rel-def-Y}) (with $g_i$ replaced by $G_i$) except the last one, which is replaced by 
\[G_i^2=1+(q-1)e_i+(q-1)e_iG_i\,,\ \ \ \ \ i=1,\dots,n-1\,.\]
To avoid confusion, we will denote this algebra by $\widetilde{Y}_{d,n}$,  and we will give an explicit isomorphism between $\widetilde{Y}_{d,n}$ and $Y_{d,n}$ later. 

\vskip .1cm
Let $z$ be another indeterminate. For convenience, we set $k:=\C(\sqrt{q},z)$. Let $c_1,\dots,c_{d-1}$ be arbitrary elements of $k$ and set $c_0:=1$. In \cite{ju2}, it is proved that there is a unique $k$-linear function $\tr$ on the chain, in $n$, of algebras $k\widetilde{Y}_{d,n}$ with values in $k$ satisfying:
\begin{equation}\label{cond-ju}
\begin{array}{lll}
(\text{C}0) & \tr(1)=1\ ,& \\[0.2em]
(\text{C}1) & \tr(xy)=\tr(yx)\ ,& \text{for any $n\geq1$ and $x,y\in \widetilde{Y}_{d,n}$;}\\[0.2em]
(\text{C}2) & \tr(xG_n)=z\tr(x)\ , & \text{for any $n\geq1$ and $x\in \widetilde{Y}_{d,n}$.}\\[0.2em]
(\text{C}3) & \tr(x t_{n+1}^b)=c_b\tr(x)\ , & \text{for any $n\geq0$, $x\in \widetilde{Y}_{d,n}$ and $b\in\{0,\dots,d-1\}$.}
\end{array}
\end{equation}
Note that here, $\widetilde{Y}_{d,n}$ is identified with a subalgebra of $\widetilde{Y}_{d,n+1}$ for any $n\geq1$.

\vskip .1cm
In \cite{jula3}, it is explained how to obtain isotopy invariants for classical and framed links from the linear function $\tr$ (see also \cite{chla,jula2} for classical and framed links and \cite{jula1} for singular links). This is done as follows.

First we take the parameters $\{c_0,c_1,\dots,c_{d-1}\}$ to be solutions of the so-called E-system \cite[Appendix]{jula2}. To do so, we fix a non-empty subset $S\subset\{1,\dots,d\}$, and we set
\begin{equation}\label{def-cb}
c_b:=\frac{1}{|S|}\sum_{a\in S}\xi_a^b\,,\ \ \ \ \ \ \ \text{for $b=0,1,\dots,d-1$,}
\end{equation}
where we recall that $\{\xi_1,\dots,\xi_d\}$ is the set of $d$-th roots of unity. We denote $\tr_S$ the unique linear function satisfying (\ref{cond-ju}) for the values (\ref{def-cb}) of the parameters $c_0,c_1,\dots,c_{d-1}$, and we set:
\begin{equation}\label{notations}
E_S:=\frac{1}{|S|}\,,\ \ \ \ \ \lambda_S:=\frac{z+(1-q)E_S}{qz}\ \ \ \ \ \text{and}\ \ \ \ \ \ D_S:=\frac{1}{\sqrt{\lambda_S}z}\,,
\end{equation}
where the field $k$ is extended by an element $\sqrt{\lambda_S}$. We denote by $k_S$ this new field.

\vskip .1cm
We let $\widetilde{\delta}^{(S)}_{Y,n}$ be the surjective morphism from $k_S\bigl[\Z/d\Z\wr B_n\bigr]$ to $k_S\widetilde{Y}_{d,n}$, defined by
\begin{equation}\label{def-tdelta}
\widetilde{\delta}^{(S)}_{Y,n}\ :\ \ \ \sigma_i\mapsto \sqrt{\lambda_S}G_i\ \ (i=1,\dots,n-1)\,,\ \ \ \ \ \ \ t_j\mapsto t_j\ \ (j=1,\dots,n)\,.
\end{equation}
The fact that $\widetilde{\delta}^{(S)}_{Y,n}$ defines indeed an algebra morphism follows from the homogeneity of the relations (\ref{def-FB}) in the braid generators $\sigma_i$.

\vskip .1cm
Finally, let $\tilde{K}$ be a $\Z/d\Z$-framed link and $\tilde{\beta}_{\tilde{K}}\in \Z/d\Z\wr B_n$ a $\Z/d\Z$-framed braid on $n$ strands having $\tilde{K}$ as its closure. Then we define the map $\FD_{Y,S}$ from the set of $\Z/d\Z$-framed links to the field $k_S$ by
\begin{equation}\label{tdelta}
\FD_{Y,S}(\tilde{K})=D_S^{n-1}\cdot\tr_S\circ\widetilde{\delta}^{(S)}_{Y,n}(\tilde{\beta}_{\tilde{K}})\ .
\end{equation}
\begin{Th}[Juyumaya--Lambropoulou \cite{jula3}]\label{theo-jula} For any $S\subset\{1,\dots,d\}$, the map $\FD_{Y,S}$ is an isotopy invariant for $\Z/d\Z$-framed links.
\end{Th}

\begin{Rem}
As in the previous  section, the invariants $\FD_{Y,S}$ can be restricted to give invariants for classical links \cite{jula2}. 
We denote $\Delta_{Y,S}$ the corresponding invariants for classical links. 
\end{Rem}

\paragraph{Comparison with invariants $\FG^{\gamma}_{Y,\rho}$.}  We keep $S$ a fixed non-empty subset of $\{1,\dots,d\}$, and $c_0,c_1,\dots,c_{d-1}$ the associated solution (\ref{def-cb}) of the E-system. In order to relate the invariant $\FD_{Y,S}$ (respectively, $\Delta_{Y,S}$) to the invariants of the form $\FG^{\gamma}_{Y,\rho}$ (respectively, $\Gamma^{\gamma}_{Y,\rho}$) obtained in Subsection \ref{subsec-inv-F}, we denote
\[\widetilde{\rho}_{S,n}\ :\ k_S\widetilde{Y}_{d,n}\to k_S\,,\ \ \ \ \widetilde{\rho}_{S,n}(x):=D_S^{n-1}\cdot\tr_S(x)\ ,\]
and define new generators by
\begin{equation}\label{def-g}
g_i:=\sqrt{\lambda_S}\bigl(\sqrt{q}+(1-\sqrt{q})e_i\bigr)G_i\,,\ \ \ \ \ \ i=1,\dots,n-1\ .
\end{equation}
Straightforward calculations show first that this change of generators is invertible since 
\begin{equation}\label{tg-g}
G_i=\sqrt{\lambda_S}^{-1}\bigl(\sqrt{q}^{-1}+(1-\sqrt{q}^{-1})e_i\bigr)g_i\,,\ \ \ \ \ \ i=1,\dots,n-1\ ,
\end{equation}
and moreover, that these new generators $g_1,\dots,g_{n-1}$ satisfy all the defining relation in (\ref{rel-def-Y}) of $Y_{d,n}$, where
\begin{equation}\label{eq-u-v}
u:=\sqrt{q\lambda_S}\ \ \ \ \ \ \ \text{and}\ \ \ \ \ \ \ v:=(q-1)\sqrt{\lambda_S}\ ,
\end{equation}
Thus, Formulas (\ref{def-g}) and (\ref{tg-g}) provide mutually inverse isomorphisms between $k_S\widetilde{Y}_{d,n}$ and $k_SY_{d,n}$, and in turn, the linear maps $\widetilde{\rho}_{S,n}$ ($n\geq1$) can be seen, via this isomorphism, as linear maps on $k_SY_{d,n}$.
We note the following formula, which is derived directly from (\ref{eq-u-v}) and (\ref{notations}):
\begin{equation}\label{eq-uv-D}
v^{-1}(1-u^2)=\frac{D_S}{|S|}\ .
\end{equation}

\begin{Prop}\label{prop-comp1}
\begin{enumerate}
\item[\rm{(i)}] The family of linear maps $\{\widetilde{\rho}_{S,n}\}_{n\geq1}$ satisfies Conditions (M1) and (M2) in (\ref{Markov-rho}), and is thus a Markov trace on $\{k_SY_{d,n}\}_{n\geq1}$.
\item[\rm{(ii)}] Moreover, we have
\begin{equation}\label{inv-jula}
\FD_{Y,S}=\FG^{\sqrt{q}^{-1}}_{Y,\tilde{\rho}_S}\ .
\end{equation}
\end{enumerate}
\end{Prop}
\begin{proof}
(i) The trace condition (M1) is obviously satisfied by the linear maps $\widetilde{\rho}_{S,n}$. From Theorem \ref{theo-jula}, it follows that the family of linear maps $\{\widetilde{\rho}_{S,n}\}_{n\geq1}$ satisfies
\[\widetilde{\rho}_{S,n+1}(xG_n)=\frac{\widetilde{\rho}_{S,n}(x)}{\sqrt{\lambda_S}}\ \ \ \text{and}\ \ \ \widetilde{\rho}_{S,n+1}(xG_n^{-1})=\sqrt{\lambda_S}\,\widetilde{\rho}_{S,n}(x)\,,\ \ \ \ \ \text{for any $n\geq1$ and $x\in Y_{d,n}$\,.}\]
It follows from Formula (\ref{def-g}) and a short calculation that $g_i^{-1}=\sqrt{\lambda_S}^{-1}\bigl(\sqrt{q}^{-1}+(1-\sqrt{q}^{-1})e_i\bigr)G^{-1}_i$. According to this and to Formula (\ref{def-g}), the Markov condition (M2) will be satisfied if
\[\widetilde{\rho}_{S,n+1}(xe_nG^{\pm1}_n)=\widetilde{\rho}_{S,n+1}(xG^{\pm1}_n)\,,\ \ \ \ \ \text{for any $n\geq1$ and $x\in Y_{d,n}$\,.}\]
The end of the proof of Theorem \ref{theo-inv} item \textbf{1}, from Relation (\ref{eq-proof}), can be repeated here.

\vskip .1cm
(ii) This is immediate in view of (\ref{tdelta}) and (\ref{tg-g}), taking into account the definition (\ref{def-tdelta}) of $\widetilde{\delta}^{(S)}_{Y,n}$.
\end{proof}

At this point, we proved that the invariants $\FD_{Y,S}$ (and thus $\Delta_{Y,S}$ as well) are included in the sets of invariants constructed in this paper. For a given $S$, to identify precisely  to which invariant $\FD_{Y,S}$ corresponds,  in view of (\ref{inv-jula}), it remains to determine the Markov trace $\{\tilde{\rho}_{S,n}\}_{n\geq1}$ in terms of the classification given in Theorem \ref{theo-mark}.

\begin{Prop}\label{prop-comp2}
 Using notations as in Theorem \ref{theo-mark}, the Markov trace $\{\widetilde{\rho}_{S,n}\}_{n\geq1}$ on $\{\tilde{k}Y_{d,n}\}_{n\geq1}$ is given by the following choice of parameters:
 \begin{equation}\label{jula-alpha}
 \alpha_{\mu^0}=\left\{\begin{array}{ll}
 0 & \text{if $\mu^0_a> 0$ for some $a\notin S$,}\\[0.4em]
 \displaystyle\frac{D_S^{|\mu^0|-1}}{|S|^{|\mu^0|}}\ \ \  & \text{otherwise.}
 \end{array}\right.
 \end{equation} 
\end{Prop}
\begin{proof} Let $\{\alpha_{\mu^0}\,,\ \mu^0\in\Comp^0_d\}$ be the set of parameters, which is to be determined, corresponding to $\{\widetilde{\rho}_{S,n}\}_{n\geq1}$. We recall that, from the classification result, the associated traces $\widetilde{\rho}_S^{\mu}$ are of the form
\[\widetilde{\rho}_S^{\mu}=\alpha_{[\mu]}\cdot\tau_{\mu_1}\otimes\dots\otimes \tau_{\mu_d}\,,\ \ \ \ \ \text{for any $\mu\in\Comp_d(n)$.}\]
For $a\in\{1,\dots,d\}$, we denote by $\alpha_a$ (respectively, $\chi_a$) the parameter (respectively, the character) corresponding to the composition $(0,\dots,0,1,0,\dots,0)$ with $1$ in $a$-th position. 

First, Condition (C3) in (\ref{cond-ju})  for $n=0$ gives
\[\widetilde{\rho}_{S,1}(t_1^b)=c_b\ ,\ \ \ \ \ \ \ \ b=0,\dots,d-1\ .\]
On the other hand, we write $t_1^b=\sum_{1\leq a \leq d}E_{\chi_a}\xi_a^b$, and we obtain
\[\widetilde{\rho}_{S,1}(t_1^b)=\sum_{1\leq a \leq d}\xi_a^b\alpha_a=c_b\ ,\ \ \ \ \ \ \ \ b=0,\dots,d-1\ .\]
Inverting the Vandermonde matrix of size $d$ with coefficients $\xi^{i-1}_j$ in row $i$ and column $j$, this yields:
\begin{equation}\label{alpha-a}\alpha_a=\frac{1}{d}\sum_{0\leq b\leq d-1}\xi^{-b}_a c_b\ ,\ \ \ \ \ \ \ \ a=1,\dots,d\ .
\end{equation}
Taking into account now the values of $c_b$ in (\ref{def-cb}) corresponding to $S$, we obtain Formula (\ref{jula-alpha}) when $|\mu^0|=1$.

Let $n>0$. Condition (C3) in (\ref{cond-ju}) now gives
\[\widetilde{\rho}_{S,n+1}(xt_{n+1}^b)=D_Sc_b\widetilde{\rho}_{S,n}(x)\ ,\ \ \ \ \ \ \ \ x\in Y_{d,n},\ \ b=0,\dots,d-1\ .\]
Let $\mu\in\Comp_d(n)$ and let $\chi_1^{\mu}$ be the character of $(\Z/d\Z)^n$ defined in (\ref{defc}). We then have, by construction,
\[\widetilde{\rho}_{S,n}(E_{\chi_1^{\mu}})=\widetilde{\rho}_S^{\mu}(1)\ ,\]
while, on the other hand, writing $E_{\chi_1^{\mu}}t_{n+1}^b=\sum_{1\leq a \leq d}\xi_a^bE_{\chi_1^{\mu^{[a]}}}$, we have
\[\widetilde{\rho}_{S,n+1}(E_{\chi_1^{\mu}}t_{n+1}^b)=\sum_{1\leq a \leq d}\xi_a^b\widetilde{\rho}_S^{\mu^{[a]}}(1)\ ,\ \ \ \ \ \ \ \ \ b=0,\dots,d-1\ .\]
We conclude that, for any $\mu\in\Comp_d(n)$ and $b=0,\dots,d-1$, we have
\[\sum_{1\leq a \leq d}\xi_a^b\alpha_{[\mu^{[a]}]}\cdot\tau_{\mu_1}(1)\dots\tau_{\mu_a+1}(1)\dots \tau_{\mu_d}(1)=D_S\,c_b\,\alpha_{[\mu]}\cdot\tau_{\mu_1}(1)\dots\tau_{\mu_a}(1)\dots \tau_{\mu_d}(1)\ .\]
Inverting the same matrix as above, and using the already obtained formula (\ref{alpha-a}), we conclude that
\[\alpha_{[\mu^{[a]}]}\cdot\tau_{\mu_1}(1)\dots\tau_{\mu_a+1}(1)\dots \tau_{\mu_d}(1)=D_S\alpha_a\alpha_{[\mu]}\cdot\tau_{\mu_1}(1)\dots\tau_{\mu_a}(1)\dots \tau_{\mu_d}(1)\,,\ \ \ \ \ \ \ \ a=1,\dots,d\,.\]
Now when $\mu_a=0$, this yields $\alpha_{[\mu^{[a]}]}=D_S\alpha_a\alpha_{[\mu]}$, which is what is needed to conclude the proof.
\end{proof}

\begin{Rem}\label{rem-traces2}
Following Remark \ref{rem-traces} after the proof of Theorem \ref{theo-mark}, we notice that the associated traces corresponding to $\{\widetilde{\rho}_{S,n}\}_{n\geq1}$ are given, for $\mu\in\Comp_d(n)$, by
\[\widetilde{\rho}^{\mu}_S=\left\{\begin{array}{ll}
 0 & \text{if $\mu_a> 0$ for some $a\notin S$,}\\[0.4em]
 \displaystyle\frac{1}{|S|}\cdot\tau_n\ \ \  & \text{otherwise.}
 \end{array}\right.\]
where $\tau_n$ acts on $\cH^{\mu}$ by restriction from $\cH_n$. This follows directly from Proposition \ref{prop-comp2}  and (\ref{eq-uv-D}).\hfill$\triangle$
\end{Rem}

\begin{Rem} Proposition \ref{prop-comp2} gives the explicit decomposition of the Markov trace $\{\widetilde{\rho}_{S,n}\}_{n\geq1}$ in the basis
$\left\{\,\{\rho_{\mu^0,n}\}_{n\geq1}\,\ |\ \ \mu^0\in\Comp^0_d\ \right\}$ and in turn, together with Proposition \ref{prop-comp1}, relates explicitly the invariant $\FD_{Y,S}$ with the invariants obtained in this paper. Concretely, we have:
\[\FD_{Y,S}= \sum_{\mu^0\in\Comp_d^0}\alpha_{\mu^0}\FG^{\sqrt{q}^{-1}}_{Y,\mu^0}\ ,\]
where the coefficients $\alpha_{\mu^0}$ are given by (\ref{jula-alpha}), and the variables $u$ and $v$ are expressed in terms of variables $q$ and $\lambda_S$ according to (\ref{eq-u-v}).\hfill$\triangle$
\end{Rem}

\vspace{1cm}
\noindent {\bf Addresses}\\

\noindent \textsc{Nicolas Jacon}, Universit\'e de Reims Champagne-Ardenne, UFR Sciences exactes et naturelles, Laboratoire de Math\'ematiques EA 4535
Moulin de la Housse BP 1039, 51100 Reims, FRANCE\\  \emph{nicolas.jacon@univ-reims.fr}\\

\noindent \textsc{Lo\"\i c Poulain d'Andecy }, Universit\'e de Reims Champagne-Ardenne, UFR Sciences exactes et naturelles, Laboratoire de Math\'ematiques EA 4535
Moulin de la Housse BP 1039, 51100 Reims, FRANCE\\  \emph{loic.poulain-dandecy@univ-reims.fr}

\end{document}